\documentclass{amsart}

\usepackage{amsmath, amsthm, amscd, amsfonts, amssymb, enumerate, verbatim, newlfont, calc, graphicx, color}
\usepackage[bookmarksnumbered, colorlinks, plainpages]{hyperref}
\usepackage{tikz}
 \usepackage{doi}

\newtheorem{theorem}{Theorem}[section]       
\newtheorem{question}[theorem]{Question}    
 
\newtheorem{lemma}[theorem]{Lemma}
\newtheorem{proposition}[theorem]{Proposition}
\newtheorem{corollary}[theorem]{Corollary}
\newtheorem{definition}[theorem]{Definition}

\newtheorem{remark}[theorem]{Remark}

\newtheorem{example}[theorem]{Example}

\usepackage{mathtools}

\usepackage{pstricks}
\usepackage{epsfig}
\usepackage{pst-grad}
\usepackage{pst-plot}
\usepackage{subfig}

\begin{document}

\author[M. Nasernejad,  V. Crispin  Q.  and   J. Toledo]{Mehrdad ~Nasernejad$^{*}$,  Veronica Crispin Qui\~n{o}nez,  Jonathan Toledo}
\title[Normally torsion-freeness and normality criteria] {Normally torsion-freeness and normality  criteria for monomial ideals}

\thanks{$^*$Corresponding author}

\thanks{E-mail addresses:  m$\_$nasernejad@yahoo.com, veronica.crispin@math.uu.se, and  jonathan.tt@itvalletla.edu.mx}  
\maketitle

\vspace{0.4cm}

\begin{abstract}
In this paper,  we focus on  the associated  primes of powers of monomial ideals  and   asymptotic behavior properties such as normally torsion-freeness, normality, the strong persistence property, and the  persistence property. In particular, we  introduce the concept of monomial ideals of well-nearly normally torsion-free type, and show that these  ideals are normal. After that, we present  some results on  the existence of   embedded associated prime ideals  in the associated primes set of powers of monomial ideals. Further,  we employ  them in investigating  the edge and cover ideals of cones of graphs.  Next, we present counterexamples to several questions concerning the relations  between relevant algebraic properties of the edge ideals of clutters and complement clutters.
 We conclude by providing counterexamples to questions on the possible connections  between normally torsion-freeness and normality of monomial ideals under polarization.  
\end{abstract}
\vspace{0.4cm}


\section{Introduction and Preliminaries}

Let $R$ be a commutative  Noetherian ring and $I$ be an ideal of $R$. A well-known result of  Brodmann \cite{BR} showed that the sequence $\{\mathrm{Ass}_R(R/I^k)\}_{k \geq 1}$ of associated prime ideals is stationary  for large $k$. In other words,  there exists a positive integer $k_0$ such that $\mathrm{Ass}_R(R/I^k)=\mathrm{Ass}_R(R/I^{k_0})$ for all $k\geq k_0$. The  minimal such $k_0$ is called the
{\it index of stability} of  $I$,  and $\mathrm{Ass}_R(R/I^{k_0})$ is called the {\it stable set } 
 of associated prime ideals of  $I$, which is denoted by $\mathrm{Ass}^{\infty }(I).$ In the context of Brodmann's result, one can ask whether it is  true that $\mathrm{Ass}_R(R/I^k)\subseteq \mathrm{Ass}_R(R/I^{k+1})$ for all $k\geq 1$? McAdam \cite{Mc} presented an example which says, in general, the  question has a negative answer. We say that an ideal $I$ of $R$  satisfies the {\it persistence property} if  $\mathrm{Ass}_R(R/I^k)\subseteq \mathrm{Ass}_R(R/I^{k+1})$ for all  $k\geq 1$.  Along this  argument,  an ideal  $I$ of $R$ has the {\it strong persistence property}  if $(I^{k+1}:_R I)=I^k$ for all  $k\geq 1$. Ratliff showed in \cite{RA} that $(I^{k+1}:_R I)=I^k$ for all large $k$. Note that the strong  persistence property implies the persistence property (for example see  \cite[Proposition 2.9]{N1}), however, the converse does  not hold (see \cite[Example 2.18]{MMV}). In the sequel, the persistence property will often be shortened as by PP and the strong equivalent by SPP.

Let   now   $I$ be  a monomial  ideal  in  a  polynomial ring $R=K[x_1,\ldots,x_n]$ over a field $K.$  In general, monomial ideals do not satisfy PP, see
 \cite{HH2}   for a counterexample.
 In order to further investigate PP for monomial ideals, their role in connecting commutative algebra with combinatorics comes to benefit, for instance, through combinatorial objects such as (hyper)graphs, simplicial complexes and posets. In particular, combinatorial methods have been applied to find classes of monomial ideals possessing PP (see \cite{N2}) as well as counterexamples in the monomial case (see \cite{KSS}  and \cite{MV}).

One of the innovators  in this area is  Villarreal, who in \cite{VI1} introduced the concept of the {\em edge ideal} associated to a simple finite graph $G$, which is denoted by $I(G)$ and is generated by the monomials $x_ix_j$ for every edge $\{x_j, x_j\}$ of $G$. Further, the Alexander dual of $I(G)$ is called the {\em cover ideal} of $G$ and is denoted by $I(G)^{\vee}$. 
 Next,   path ideals of graphs were first introduced by Conca  and De Negri \cite{CD} in the context of  monomial ideals of linear type. 
  So,  we have the following subquestion:

  Which classes of monomial ideals have the persistence property?\\
  The answer is affirmative, when $I$ is the cover ideal of a perfect graph (see \cite[Corollary 5.11]{FHV}) or the edge ideal of a finite graph, simple or with loops (see \cite[Theorem 2.15]{MMV} and \cite[Corollary 2]{RT}, respectively).  Moreover, any polymatroidal ideal has PP (see \cite[Proposition 3.3]{HRV}).

  There are two more concepts related to PP. The first one was presented in  \cite{RT}: an  ideal $I$ in a commutative Noetherian ring $R$ has the  {\it symbolic strong persistence property} if $(I^{(k+1)}:_R I^{(1)})=I^{(k)}$ for all $k\geq 1$, where $I^{(k)}=\bigcap_{\mathfrak{p}\in \mathrm{Min}(I)}(I^kR_\mathfrak{p}\cap R)$ denotes the  $k$-th symbolic  power  of $I$. 
  In the same paper the authors show that this property follows from SPP. However, little is known about other classes of monomial ideals which satisfy the symbolic SPP, see  \cite{NKRT}. 

The second related concept is the following: an ideal $I$ in a commutative Noetherian ring $R$ is called {\it normally torsion-free} if $\mathrm{Ass}(R/I^k)\subseteq \mathrm{Ass}(R/I)$ 
  for all $k\geq 1$.  Examples of normally torsion-free monomial ideals appear from graph theory.   Already in 1994, it was proved in  \cite{SVV} 
   that  a finite simple graph $G$ is bipartite if and only if its edge ideal is normally torsion-free.  Later it was shown in  \cite{GRV} 
    that the cover ideals of bipartite graphs are normally torsion-free.  Another example of the connection to graph theory is that 
   normally torsion-free square-free monomial ideals are closely related to Mengerian hypergraphs.    
   Although normally  torsion-free square-free monomial  ideals  have been 
    studied  in \cite{KHN1, NQH, SU},  little is known for the   normally torsion-free monomial ideals which are not square-free, for example,  see \cite{OL}.

We now turn our attention to the notion of the normality of ideals.  An element $f\in R$ is {\it integral} over an ideal $I\subset R$, if  there exists an equation  $$f^k+c_1f^{k-1}+\cdots +c_{k-1}f+c_k=0 ~~\mathrm{with} ~~ c_i\in I^i.$$
The {\em integral closure} of $I$ is $\bar I$, the set of the elements in $R$ which are integral over $I$.
 The ideal $I$ is  called {\it integrally closed}, if $I=\overline{I}$, and $I$ is said to be  {\it normal} if all powers of $I$ are integrally closed, see \cite{SH} for more information. 
  If   $I$ is  a monomial ideal in a polynomial ring $R$, then  $\overline{I}$ is the monomial ideal generated by all monomials $u \in R$ for which there exists an integer $k$ such that   $u^{k}\in I^{k}$, refer to    \cite[Theorem 1.4.2]{HH1}.  \par

  The last related definition for a monomial ideal in $R$ is as follows: $I$ is called   {\it nearly normally torsion-free}  if there exist a positive integer $k$ and a monomial prime ideal  $\mathfrak{p}$ such that $\mathrm{Ass}_R(R/I^m)=\mathrm{Min}(I)$ for all $1\leq m\leq k$, and
 $\mathrm{Ass}_R(R/I^m) \subseteq \mathrm{Min}(I) \cup \{\mathfrak{p}\}$ for all $m \geq k+1$, where  $\mathrm{Min}(I)$  denotes  the set of minimal prime ideals of $I$, see \cite[Definition 2.1]{Claudia}. For some classes of nearly normally torsion-free monomial ideals, see \cite{NBR, NQBM}. 
 
 One important implication is that a normal monomial ideal has SPP by  \cite[Theorem 6.2]{RNA}.  
   This result is used in \cite{NKA}, where SPP is shown  through normality in the cover ideals of the following imperfect graphs: (1) cycle graphs of odd orders, (2) wheel graphs of even orders, and (3) Helm graphs of odd orders with greater than or equal to $5$. 
  In particular,  more information related to the normality of the cover ideals of imperfect graphs  can be found in \cite{ANR1, ANR2, NBR}.
 Moreover, based on \cite[Theorem 1.4.6]{HH1}, every normally torsion-free square-free monomial ideal is normal. Hence, in the case of square-free monomial ideal, we have: normally torsion-freeness $\Rightarrow$   normality $\Rightarrow$ 
  strong persistence property  $\Rightarrow$  persistence property.  
 
 The main aim of this work  is to continue these research lines and identify  the other classes of monomial ideals that satisfy PP, SPP,  normality, and normally torsion-freeness. In particular, we present  some criteria for normally torsion-freeness and normality of  monomial ideals. In addition, we obtain  some results related to the appearance of the graded  maximal ideal in the associated primes of powers of monomial ideals.

Throughout this paper, $R=K[x_1, \ldots, x_n]$ stands for a polynomial  ring over a field $K$, unless stated otherwise.  Also, $\mathcal{G}(I)$ denotes the unique minimal set of monomial generators of a monomial ideal $I\subset R$.  
The {\em support} of a monomial $u\in R$, denoted by $\mathrm{supp}(u)$, is the set of variables that divide $u$. For a monomial ideal $I$, we set $\mathrm{supp}(I)=\bigcup_{u \in \mathcal{G}(I)}\mathrm{supp}(u)$.

The structure of this paper is outlined as follows: In Section 2, we establish a connection between the presence of a prime monomial ideal in the associated
prime set of powers of a monomial ideal and the cardinality of a maximal independent set of the monomial ideal, as demonstrated in Corollary \ref{Cor.1}.
Subsequently, our attention shifts to examining the normality of monomial ideals within a subset of nearly normally torsion-free square-free monomial
ideals known as well-nearly normally torsion-free type. In particular, we present the following theorem:

 \bigskip
 \textbf{Theorem  2.12.}
 Let $I\subset R=K[x_1, \ldots, x_n]$  be   of  well-nearly  normally torsion-free  type. Then   $I$ is normal, and so has (symbolic) (strong) persistence property.
 
\bigskip  
 Section 3  is concerned with  the existence of   embedded associated prime ideals  in the associated primes set of powers of monomial ideals, 
 see Propositions    \ref{Pro.Maximal1} and   \ref{Pro.Maximal2}.  In particular, we show that these consequences can be employed in studying the edge and cover ideals of cones of graphs as follows: 
 
 \bigskip
\textbf{Lemma 3.5.}
 Let   $G$  be  a  connected graph  and    $H:=C(G)$  be  its cone. Then the following statements hold:
 \begin{itemize}
 \item[(i)] If   $(x_i : i\in V(G))\in \mathrm{Ass}(J(G)^t)$ for some $t\geq 1$,  then  $(x_j : j\in V(H))\in \mathrm{Ass}(J(H)^{t+1})$, 
 where $J(G)$  and $J(H)$ denote the cover ideals of $G$ and $H$, respectively. 
 
 \item[(ii)]  If  $(x_i : i\in V(G))\in \mathrm{Ass}(I(G)^t)$  for some  $t\geq 2$,  then   $(x_j : j\in V(H))\in \mathrm{Ass}(I(H)^{t})$,  
 where $I(G)$   and $I(H)$ denote the edge ideals of $G$ and $H$, respectively. 
 \end{itemize}
  
Now, let $I$ be  a normally torsion-free square-free  monomial ideal with $\mathcal{G}(I)=\{u_1, \ldots, u_m\}$. Let us  delete one of the members of $\mathcal{G}(I)$, say $u_m$, and add a new square-free monomial, say $v$. There is no guarantee that  the new square-free monomial
 ideal $(u_1, \ldots, u_{m-1})+(v)$ is still normally torsion-free.  
 We close this section with giving a  result, which says that  under certain conditions the square-free monomial ideal $(u_1, \ldots, u_{m-1})+(v)$  is normally torsion-free as well, refer to Theorem   \ref{NTF3}.

Section 4  is devoted to  comparing  some algebraic properties of the edge ideals of  clutters and complement clutters, see  Question  \ref{Complement-Clutter}.
 Section 5  deals with the  normally torsion-freeness and normality of monomial ideals under the polarization operator, refer to Question \ref{POLARIZATION-NTF}. 


\section{Normality of nearly normally torsion-free  monomial ideals}

In this section, we explore the relationship between the presence of a prime monomial ideal in the associated prime set of powers of a monomial ideal and
the size of a maximal independent set of the monomial ideal, as elucidated in Corollary \ref{Cor.1}. Following this, we investigate the normality of monomial ideals
within a specific class comprising nearly normally torsion-free square-free monomial ideals, as outlined in Theorem \ref{Normal+Well-NNTF}. Additionally, we will utilize the following proposition, which serves as an extension of  \cite[Theorem 3.5]{HM}  and \cite[Theorem 3.4]{SNQ}, in the proof of 
 Proposition \ref{Lem.Distinct-Variables}.

 To capture  Proposition  \ref{Lem.Distinct-Variables} and Corollary \ref{Cor.1},  we need to recollect  some definitions. Assume that $I$ is  a square-free monomial ideal and $\Gamma \subseteq \mathcal{G}(I)$, where $\mathcal{G}(I)$ denotes the unique minimal set of monomial generators of the  monomial ideal $I$. We say that $\Gamma$ is an {\it independent} set in $I$ if $\mathrm{gcd}(f,g)=1$ for each $f,g\in \Gamma$ with $f\neq g$. We denote the maximum cardinality of an independent set in $I$ by $\beta_1(I)$. 
Furthermore, if $I$ is a monomial ideal, then the  {\it deletion} of $x_i$  from  $I$ with $1\leq i \leq n$, denoted by $I\setminus x_i$, is obtained by setting $x_i=0$ in every minimal generator of $I$, that is, we delete every minimal generator such as $u\in \mathcal{G}(I)$ with $x_i\mid u$, refer to 
  \cite[page 303]{HM} for  definitions of deletion and independent set.

\begin{proposition}  \label{Pro.Distinct-Variables}
Let $I\subset R=K[x_1, \ldots, x_n]$ be a  monomial ideal,  $t$ a positive integer, $\mathfrak{p}$ a prime monomial ideal in $R$,
  and $y_1, \ldots, y_s$ be distinct variables in $R$  such that, 
for each  $i=1, \ldots, s$, $\mathfrak{p}\setminus y_i \notin \mathrm{Ass}(R/(I\setminus y_i)^t)$, where $I\setminus y_i$ denotes the 
deletion  of $y_i$  from  $I$.  Then  $\mathfrak{p}\in \mathrm{Ass}(R/I^t)$ if and only if $\mathfrak{p}\in \mathrm{Ass}(R/(I^t:\prod_{i=1}^sy_i))$. 
\end{proposition} 
\begin{proof}
($\Leftarrow$) Let $\mathfrak{p}\in \mathrm{Ass}(R/(I^t:\prod_{i=1}^sy_i))$. In view of  \cite[Exercise 2.1.62]{V1}, we have  the following short exact sequence
$$0\longrightarrow R/(I^t:\prod_{i=1}^sy_i)\stackrel{\psi}{\longrightarrow} R/I^t \stackrel{\phi}{\longrightarrow} R/(I^t, \prod_{i=1}^sy_i) \longrightarrow  0,$$
where $\psi(\overline{r})=\overline{r}\prod_{i=1}^sy_i$ and $\phi(\overline{r})=\overline{r}$. It follows also from    \cite[Exercise 9.42]{sharp} that 
\begin{equation} \label{12}
\mathrm{Ass}(R/(I^t:\prod_{i=1}^sy_i))\subseteq \mathrm{Ass}(R/I^t) \subseteq \mathrm{Ass}(R/(I^t:\prod_{i=1}^sy_i)) \cup \mathrm{Ass}(R/(I^t, \prod_{i=1}^sy_i)). 
\end{equation}
By virtue of   $\mathfrak{p}\in \mathrm{Ass}(R/(I^t:\prod_{i=1}^sy_i))$, we deduce that   
$\mathfrak{p}\in \mathrm{Ass}(R/I^t)$. 

($\Rightarrow$) Assume that  $\mathfrak{p}\in \mathrm{Ass}(R/I^t)$.  To show  $\mathfrak{p}\in \mathrm{Ass}(R/(I^t:\prod_{i=1}^sy_i))$, we use induction  on $s$.  
Let $J:=I\setminus y_1$. Since  the generators of $J^t$ are precisely the generators of $I^t$ that are not divisible by $y_1$, we get  $(I^t,y_1)=(J^t,y_1)$. One can deduce  from our  assumption  that  $\mathfrak{p}\setminus y_1 \notin \mathrm{Ass}(R/J^t)$. Also, it follows from \cite[Theorem 5.5]{KHN2}
  that  $\mathfrak{p}\in \mathrm{Ass}(R/(J^t, y_1))$ if and only if $\mathfrak{p}=(\mathfrak{p}_1, y_1)$, 
where  $\mathfrak{p}_1 \in \mathrm{Ass}(R/J^t)$.
 This implies that  $\mathfrak{p}\notin \mathrm{Ass}(R/(J^t, y_1))$, and so   $\mathfrak{p}\notin \mathrm{Ass}(R/(I^t, y_1))$. By  virtue of (\ref{12}), we have   $\mathrm{Ass}(R/I^t) \subseteq \mathrm{Ass}(R/(I^t:y_1)) \cup \mathrm{Ass}(R/(I^t, y_1))$, and hence   
 $\mathfrak{p}\in \mathrm{Ass}(R/(I^t:y_1))$. 
 Consequently,  the claim is true for the case in which $s=1$. Now, assume  that the assertion has been shown  for a product of $s-1$ variables, and that $\mathfrak{p}\in \mathrm{Ass}(R/I^t)$. Set $M:=\prod_{i=1}^{s-1}y_i$. It follows from  the induction that  
$\mathfrak{p}\in \mathrm{Ass}(R/(I^t:M))$. Moreover,    
\begin{equation} \label{13}
 \mathrm{Ass}(R/(I^t:M))\subseteq \mathrm{Ass}(R/((I^t:M):y_s) )\cup \mathrm{Ass}(R/((I^t:M),y_s)).
\end{equation} 
Let  $L:=I\setminus y_s$. We show that  $((I^t:M),y_s)=((L^t:M),y_s).$ Since $L\subseteq I$, it is easy to see that 
  $((L^t:M),y_s) \subseteq ((I^t:M),y_s).$ To establish  the reverse inclusion,  pick  a monomial $u$ in 
$((I^t:M),y_s)$. If $y_s\mid u$, then $u\in  ((L^t:M),y_s),$ and the  proof  is over. If $y_s\nmid u$, then 
$u\in  (I^t:M),$ and hence $uM\in I^t$. This yields that there exists a monomial $f\in \mathcal{G}(I^t)$ such that $f\mid uM$. 
 On account of   $y_s\nmid u$ and $y_s \nmid M$, this gives  that $y_s\nmid f$, and so $f\in L^t$. Hence,  $u\in (L^t:M)$, and thus  $((I^t:M),y_s) \subseteq ((L^t:M),y_s).$ 
The  assumption implies that $\mathfrak{p}\setminus y_s \notin \mathrm{Ass}(R/L^t)$. As  $\mathrm{Ass}(R/(L^t:M)) \subseteq \mathrm{Ass}(R/L^t)$, one obtains   $\mathfrak{p}\setminus y_s \notin \mathrm{Ass}(R/(L^t:M))$. In  light of   \cite[Theorem 5.5]{KHN2}, we have    
 $\mathfrak{p}\in \mathrm{Ass}(R/((L^t:M), y_s))$  if and only if $\mathfrak{p}=(\mathfrak{p}_1, y_s)$, 
where  $\mathfrak{p}_1 \in \mathrm{Ass}(R/(L^t:M))$.  This leads  to 
$\mathfrak{p}\notin \mathrm{Ass}(R/((L^t:M), y_s))$, and so $\mathfrak{p}\notin\mathrm{Ass}(R/((I^t:M), y_s)).$ 
From    $(\ref{13})$, this yields  that   $\mathfrak{p}\in \mathrm{Ass}(R/((I^t:M): y_s))=\mathrm{Ass}(R/(I^t:\prod_{i=1}^sy_i))$. 
 This   completes the inductive  step, and thus   the claim has been proved by induction. 
\end{proof}


We are ready to state the first main result of this section.  

\begin{proposition}   \label{Lem.Distinct-Variables}
  Let   $I \subset R$ be a   monomial ideal, $\ell$ a positive integer, $v$ a square-free monomial in $R$ with $v\in I^\ell$,  and  $\mathfrak{p}$ a prime monomial ideal  such that  $\mathfrak{p}\setminus x_i \notin \mathrm{Ass}(R/(I\setminus x_i)^s)$ for all   $x_i\in \mathrm{supp}(v)$ and some positive integer $s$.  If $\mathfrak{p}\in \mathrm{Ass}(R/I^s)$, then $s>\ell$. 
 \end{proposition}

\begin{proof}
 Let $\mathfrak{p}\in \mathrm{Ass}(R/I^s)$.  Without loss of generality,  assume that $v=\prod_{i=1}^rx_i$. It follows from  
  Proposition \ref{Pro.Distinct-Variables} that    $\mathfrak{p}\in \mathrm{Ass}(R/(I^s:\prod_{i=1}^rx_i)).$    If $s\leq \ell$, then we deduce that 
     $(I^s:\prod_{i=1}^rx_i)=R$ since $v\in I^\ell$, which contradicts the fact that  $\mathfrak{p}\in \mathrm{Ass}(R/(I^s:\prod_{i=1}^rx_i)).$ This yields that  $s >\ell$, as claimed.   
 \end{proof}


\begin{corollary} \label{Cor.1}
 Let $I\subset  R$ be a square-free  monomial ideal,  $\mathfrak{p} \subset R$ be  a prime monomial ideal,  and  $\{u_1, \ldots, u_{\beta_1(I)}\}$  be a maximal independent set of minimal generators of $I$ such that  $\mathfrak{p}\setminus x_i \notin \mathrm{Ass}(R/(I\setminus x_i)^s)$ for all    $x_i\in \mathrm{supp}(\prod_{i=1}^{\beta_1(I)}u_i)$ and some positive integer $s$.  If $\mathfrak{p}\in \mathrm{Ass}(R/I^s)$, then $s\geq \beta_1(I)+1$. 
 \end{corollary}
 

  In view of  Corollary \ref{Cor.1},  one may ask the following.

\begin{question} \label{Q.Distinct-Variables1}
 Let   $I \subset R$  be a  square-free  monomial ideal, $\mathfrak{p}$ a prime monomial ideal in $R$,   and  $\{u_1, \ldots, u_{\beta_1(I)}\}$  be a maximal independent set of minimal generators of $I$ such that  $\mathfrak{p}\setminus x_i \notin \mathrm{Ass}(R/(I\setminus x_i)^t)$ for all   $x_i\in \bigcup_{i=1}^{\beta_1(I)}\mathrm{supp}(u_i)$ and some positive integer $t$.  
 \begin{itemize}
 \item[(i)]  If  $t\geq \beta_1(I)+1$, then can we deduce that   $\mathfrak{p}\in \mathrm{Ass}(R/I^t)$?
\item[(ii)] If $\mathfrak{p}\in \mathrm{Ass}(R/I^t)$, then is it possible the  inequality $t\geq \beta_1(I)+1$  is strict?
 \end{itemize}
  \end{question}
  
  Before presenting the counterexample in the upcoming example, we first revisit the following result.

\begin{proposition}\label{Intersection}  (\cite[Proposition 3.3]{NQ})
Suppose that  $I$ and  $J$  are   two  normally torsion-free square-free monomial ideals in  $R$  such that  $\mathrm{supp}(I) \cap \mathrm{supp}(J)=\emptyset$.  Then $I\cap J=IJ$ is normally torsion-free.  
\end{proposition}

\begin{example} {\em 
Let $I\subset K[x_1, \ldots, x_{11}]$ be  the following square-free monomial ideal 
\begin{align*}
I=(& x_1, x_2) \cap (x_2, x_3)\cap (x_3, x_4)\cap (x_4, x_5) \cap (x_5, x_1)  \cap (x_1, x_7)\cap (x_2, x_8)\cap (x_3, x_9) \\
&  \cap (x_4, x_{10})\cap (x_5, x_{11})\cap (x_6, x_1)\cap (x_6, x_2) \cap (x_6, x_3)\cap (x_6, x_4) \cap (x_6, x_5)\\
 = (& x_1x_2x_3x_4x_5,x_1x_2x_4x_6x_9x_{11},x_1x_3x_4x_6x_8x_{11},  x_1x_2x_3x_4x_6x_{11},x_1x_3x_5x_6x_8x_{10}, \\ 
& x_2x_3x_5x_6x_7x_{10},x_1x_2x_3x_5x_6x_{10}, x_2x_4x_5x_6x_7x_9,x_1x_2x_4x_5x_6x_9,x_1x_3x_4x_5x_6x_8, \\ 
& x_2x_3x_4x_5x_6x_7). 
 \end{align*}
  In fact, $I$ is the cover ideal of the below  Helm graph $G$.
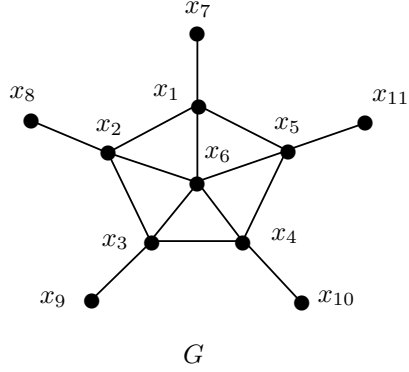
  \begin{center}
\scalebox{1} 
{
\begin{pspicture}(0,-2.4629688)(6.022813,2.4629688)
\psdots[dotsize=0.2](2.7290626,0.98609376)
\psdots[dotsize=0.2](3.9090626,0.38609374)
\psdots[dotsize=0.2](1.5290625,0.36609375)
\psdots[dotsize=0.2](2.7090626,-0.03390625)
\psdots[dotsize=0.2](3.3090627,-0.81390625)
\psdots[dotsize=0.2](2.1090627,-0.81390625)
\psline[linewidth=0.024cm](2.7090626,1.0060937)(1.5090624,0.36609375)
\psline[linewidth=0.024cm](2.7090626,1.0060937)(3.8890626,0.40609375)
\psline[linewidth=0.024cm](3.8890626,0.40609375)(3.3090627,-0.7939063)
\psline[linewidth=0.024cm](2.6890626,0.02609375)(2.6690626,-0.03390625)
\psline[linewidth=0.024cm](2.7090626,-0.01390625)(2.7090626,1.0060937)
\psline[linewidth=0.024cm](2.6890626,0.02609375)(2.7290626,0.02609375)
\psline[linewidth=0.024cm](2.6890626,-0.01390625)(3.9090626,0.40609375)
\psline[linewidth=0.024cm](2.6890626,-0.01390625)(1.5490625,0.36609375)
\psline[linewidth=0.024cm](1.5490625,0.36609375)(1.5090624,0.36609375)
\psline[linewidth=0.024cm](1.5090624,0.40609375)(2.0890625,-0.7939063)
\psline[linewidth=0.024cm](2.0890625,-0.7939063)(2.7090626,-0.03390625)
\psline[linewidth=0.024cm](2.7090626,-0.01390625)(3.2890625,-0.77390623)
\psdots[dotsize=0.2](4.9290624,0.76609373)
\psdots[dotsize=0.2](4.089062,-1.6139063)
\psdots[dotsize=0.2](1.3090625,-1.5939063)
\psdots[dotsize=0.2](0.50906247,0.7860938)
\psdots[dotsize=0.2](2.7090626,1.9460937)
\psline[linewidth=0.024cm](2.7090626,1.9660937)(2.7090626,1.0060937)
\psline[linewidth=0.024cm](3.8890626,0.40609375)(4.9290624,0.76609373)
\psline[linewidth=0.024cm](4.089062,-1.6139063)(3.2890625,-0.75390625)
\psline[linewidth=0.024cm](1.2890625,-1.5539062)(1.2890625,-1.5939063)
\psline[linewidth=0.024cm](1.3090625,-1.5739062)(2.0890625,-0.81390625)
\psline[linewidth=0.024cm](0.4890625,0.80609375)(1.4890623,0.38609374)
\usefont{T1}{ptm}{m}{n}
\rput(2.3032813,1.1760937){$x_1$}
\usefont{T1}{ptm}{m}{n}
\rput(1.5432812,0.7160938){$x_2$}
\usefont{T1}{ptm}{m}{n}
\rput(1.6232813,-0.80390626){$x_3$}
\psline[linewidth=0.024cm](2.0895312,-0.7954686)(3.2695312,-0.7954686)
\usefont{T1}{ptm}{m}{n}
\rput(3.8623438,-0.72546864){$x_4$}
\usefont{T1}{ptm}{m}{n}
\rput(3.9023438,0.73453134){$x_5$}
\usefont{T1}{ptm}{m}{n}
\rput(2.9823437,0.35453135){$x_6$}
\usefont{T1}{ptm}{m}{n}
\rput(2.7223437,2.2745314){$x_7$}
\usefont{T1}{ptm}{m}{n}
\rput(0.40234375,1.1545314){$x_8$}
\usefont{T1}{ptm}{m}{n}
\rput(0.8023437,-1.5654685){$x_9$}
\usefont{T1}{ptm}{m}{n}
\rput(4.552344,-1.5454687){$x_{10}$}
\usefont{T1}{ptm}{m}{n}
\rput(5.272344,1.1145314){$x_{11}$}
\usefont{T1}{ptm}{m}{n}
\rput(2.6623437,-2.2854688){$G$}
\end{pspicture} 
}
\end{center}
 
 \bigskip
 
  It is not hard to see that  $\{u_1= x_1x_2x_3x_4x_5\}$  is  a maximal independent set of minimal generators of $I$, and so $\beta_1(I)=1$. Using {\it Macaulay2} \cite{GS}, we can derive that 
 $$\mathfrak{p}:=(x_1, x_2, x_3, x_4, x_5, x_6)\in \mathrm{Ass}(R/I^3)\setminus(\mathrm{Ass}(R/I)\cup \mathrm{Ass}(R/I^2)).$$
  On the other hand, direct computations give the following table:
 
 \begin{center}
 \begin{tabular}{|c|}
\hline
 $I\setminus x_1=(x_2) \cap  (x_3, x_4) \cap (x_5)  \cap (x_7)\cap (x_3, x_9)  \cap (x_4, x_{10})\cap (x_6)$\\
\hline
 $I\setminus x_2= (x_1) \cap (x_3)\cap  (x_8) \cap (x_4, x_5)   \cap (x_4, x_{10})\cap (x_5, x_{11})\cap (x_6)$\\
\hline
$I\setminus x_3=  (x_2)\cap (x_4) \cap (x_5, x_1)  \cap (x_1, x_7)\cap (x_9)   \cap (x_5, x_{11}) \cap (x_6)$\\
\hline
 $I\setminus x_4= (x_1, x_2) \cap (x_3)\cap (x_5) \cap (x_1, x_7)\cap (x_2, x_8)  \cap (x_{10})\cap  (x_6)$\\
\hline
 $I\setminus x_5=  (x_2, x_3)\cap (x_4) \cap (x_1)  \cap (x_2, x_8)\cap (x_3, x_9) \cap (x_{11})\cap (x_6)$\\
\hline
\end{tabular}
\end{center}
We verify that $I\setminus x_1$ is normally torsion-free. Since $ (x_3, x_4) \cap (x_3, x_9)  \cap (x_4, x_{10})$ can be viewed as the cover ideal of  the path graph $P$ with $V(P)=\{x_3, x_4, x_9, x_{10}\}$ and $E(P)=\{\{x_3, x_4\}, \{x_3, x_9\}, \{x_4, x_{10}\}\}$, and by  this fact that the cover ideal of any bipartite graph is normally torsion-free, we deduce that $ (x_3, x_4) \cap (x_3, x_9)  \cap (x_4, x_{10})$ is normally torsion-free. Now, 
Proposition \ref{Intersection} implies that $I\setminus x_1$ is normally torsion-free.  A similar argument can be used to show that $I\setminus x_i$ is normally torsion-free for all $i=2, \ldots, 5$. In particular, we can derive that $\mathfrak{p}\setminus x_i \notin \mathrm{Ass}((I\setminus x_i)^2)$ for all $i=1, \ldots, 5$.  Now, for  $t=2= \beta_1(I)+1$, we have  $\mathfrak{p}\notin \mathrm{Ass}(R/I^2)$, that is,  the converse of   Proposition  \ref{Lem.Distinct-Variables}  is not true in general.  This refutes  (i). 
On the other hand, it follows from  $\mathfrak{p}\in \mathrm{Ass}(R/I^3)\setminus  \mathrm{Ass}(R/I^2)$ that we must have $t>2$. Consequently,  the  inequality $t\geq \beta_1(I)+1$ in  Proposition   \ref{Lem.Distinct-Variables}  can be  strict, which is a positive answer to (ii). 
}
 \end{example}


 Along with Question \ref{Q.Distinct-Variables1}, we can ask the following.

\begin{question} \label{Q.NNTF3}
Let $I$  be a nearly  normally torsion-free square-free monomial ideal in   $R$   such that  there exist  some $m\geq 1$ and some prime monomial ideal $\mathfrak{q}$  such that $\mathrm{Ass}(I^t)=\mathrm{Min}(I)$ for all $1\leq t \leq m$, and $\mathrm{Ass}(I^t) \subseteq \mathrm{Min}(I) {\cup} \{\mathfrak{q}\}$ for all $t\geq m+1$.  Let also  $\{u_1, \ldots, u_{\beta_1(I)}\}$  be a maximal independent set of minimal generators of $I$.  If $\mathfrak{q}\in  \mathrm{Ass}(R/I^s)$ for some $s\geq m+1$, then can we deduce that  $s\geq \beta_1(I)+1$?
\end{question}

Next, we describe an ideal that provides a negative answer.
 
 \begin{example}
{\em 
 Let  $I\subset  K[x_1, \ldots, x_7]$ be the following square-free monomial ideal 
 \begin{align*}
 I=(&x_1, x_2, x_3) \cap (x_2, x_3, x_4) \cap (x_3,x_4, x_5) \cap (x_4, x_5, x_6)   \cap (x_5, x_6, x_7)\\
 &  \cap (x_6, x_7, x_1) \cap (x_7, x_1, x_2)\\
 = (&x_3x_6x_7,x_3x_5x_7,x_2x_5x_7,x_3x_4x_7,x_2x_4x_7,x_1x_4x_7,x_2x_5x_6,x_2x_4x_6,x_1x_4x_6, \\ 
 & x_2x_3x_6,x_1x_3x_6,x_1x_4x_5,x_1x_3x_5,x_1x_2x_5).  
  \end{align*}
 It is a common observation that the ideal $I$ dominates the odd cycle graph  $C_7$ with  
  $$E(C_7)=\{\{x_i, x_{i+1}\} : i=1, \ldots, 6\} \cup \{\{x_1, x_7\}\}.$$
 From \cite[Theorem 4.3]{NQ},   and in view of the proof of  \cite[Lemma 3.2]{NQKR}, we can deduce that $I$ is nearly normally torsion-free such that $$\mathrm{Ass}(I^s)=\mathrm{Min}(I) \cup \{(x_1, x_2, x_3, x_4, x_5, x_6, x_7)\}  \; \text{ for all } \; s\geq 2.$$   
  In addition, one can easily check    that  $\{u_1= x_1x_2x_5,  u_2=x_3x_6x_7\}$  is  a maximal independent set of minimal generators of $I$, and so $\beta_1(I)=2$.  With the notation of Question \ref{Q.NNTF3}, we detect that $m=1$, $(x_1, x_2, x_3, x_4, x_5, x_6, x_7)\in \mathrm{Ass}(I^2)$, while $2\ngeq \beta_1(I)+1$. 
}
\end{example}


In the rest of this section,   we focus on a  class of nearly  normally torsion-free square-free monomial ideals and show that every ideal in this class is normal.
 We start with  the following definition.

   \begin{definition}
  Let $I\subset R$  be a square-free monomial ideal. Then we say that $I$ is of  {\it well-nearly normally torsion-free type}  if  
  there exist  some $\ell, m\geq 1$ and some prime monomial ideal $\mathfrak{q}$  such that $\mathrm{Ass}(I^t)=\mathrm{Min}(I)$ for all $1\leq t \leq m$, and  $\mathrm{Ass}(I^t) \subseteq \mathrm{Min}(I) {\cup} \{\mathfrak{q}\}$ for all $t\geq m+1$ such that  if $\mathfrak{q}\in \mathrm{Ass}(I^s)$ for some $s\geq m+1$, then   $I^s=\bigcap_{\mathfrak{p}\in \mathrm{Min}(I)}\mathfrak{p}^s  \cap \mathfrak{q}^{\ell s}$   is a minimal primary decomposition of $I^s$. 
   \end{definition}
   
   
We demonstrate  a large class of  monomial ideals  of  well-nearly normally torsion-free type. 
  
 \begin{example}     {\em 
 Suppose that   $C_{2n+1}$ is  an odd cycle graph with the  edge set 
$$E(C_{2n+1})=\{\{x_i, x_{i+1}\} : i=1, \ldots, 2n\} \cup \{\{x_1, x_{2n+1}\}\}.$$
It has already been proved in \cite[Lemma 3]{KNT} that,  for all $s\geq 2$, we have  
$$J(C_{2n+1})^s= \bigcap_{i=1}^{2n+1}(x_i, x_{i+1})^s \cap \mathfrak{m}^{s(n+1)},$$ 
where $\mathfrak{m}=(x_1, \ldots, x_{2n+1})$ is the maximal ideal  in   $R=K[x_1, \ldots, x_{2n+1}]$. Notice that  $\mathfrak{q}=\mathfrak{m}$ and $\ell=n+1$. In other words, the cover ideals of odd cycle graphs are  of  well-nearly normally torsion-free type. 
}
 \end{example}

   
   Next, we are going to give a square-free monomial ideal  which is nearly 
   normally torsion-free, but  is not of well-nearly normally torsion-free type.

 \begin{example}{\em 
 Let  $S=K[x_1, x_2, x_3, x_4, x_5]$ and $R=S[x_6]$. Also, let $I\subset R$  be  the following square-free monomial ideal
 \begin{align*}
 I=&(x_1x_3x_5,x_1x_3x_4,x_1x_2x_4,x_2x_4x_5x_6,x_2x_3x_5x_6)\\
=& (x_1,x_2)  \cap  (x_2, x_3)  \cap  (x_3, x_4) \cap  (x_4, x_5) \cap (x_5, x_1) \cap (x_1, x_6).
 \end{align*}
 We claim that $I$ is nearly normally torsion-free and also is normal. 
 One can easily see that $I$ is the cover ideal of the following graph $G$. 
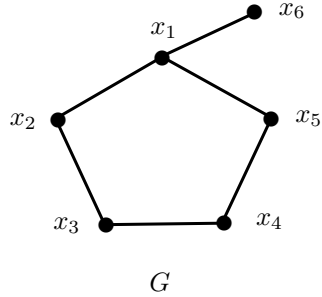
 \begin{center}
 \bigskip
\scalebox{1} 
{
\begin{pspicture}(0,-1.9729687)(4.6628127,1.9729687)
\psdots[dotsize=0.2](2.2409375,1.1545312)
\psdots[dotsize=0.2](1.5009375,-1.0454688)
\psdots[dotsize=0.2](3.0609374,-1.0254687)
\psdots[dotsize=0.2](3.6809375,0.35453126)
\psdots[dotsize=0.2](0.8609375,0.33453125)
\psdots[dotsize=0.2](3.4609375,1.7745312)
\psline[linewidth=0.04cm](2.2209375,1.1745312)(3.4609375,1.7545313)
\psline[linewidth=0.04cm](0.8609375,0.37453124)(2.2409375,1.1745312)
\psline[linewidth=0.04cm](2.2409375,1.1745312)(3.6809375,0.37453124)
\psline[linewidth=0.04cm](0.8409375,0.33453125)(1.4809375,-1.0454688)
\psline[linewidth=0.04cm](1.4809375,-1.0454688)(3.0809374,-1.0254687)
\psline[linewidth=0.04cm](3.0409374,-1.0054687)(3.6809375,0.35453126)
\usefont{T1}{ptm}{m}{n}
\rput(2.2623436,1.5245312){$x_1$}
\usefont{T1}{ptm}{m}{n}
\rput(0.40234375,0.32453126){$x_2$}
\usefont{T1}{ptm}{m}{n}
\rput(0.98234373,-1.0554688){$x_3$}
\usefont{T1}{ptm}{m}{n}
\rput(3.6623437,-0.99546874){$x_4$}
\usefont{T1}{ptm}{m}{n}
\rput(4.182344,0.36453125){$x_5$}
\usefont{T1}{ptm}{m}{n}
\rput(3.9623437,1.7845312){$x_6$}
\usefont{T1}{ptm}{m}{n}
\rput(2.2123437,-1.7954688){$G$}
\end{pspicture} 
}
\end{center}
It follows from \cite[Theorem 2.5]{NKA} that  $\mathrm{Ass}_{R}(R/J(G)^s)=\mathrm{Ass}_{S}(S/J(C_5)^s)\cup
\{(x_1, x_6)\}$ for all  $s\geq 1$, where $J(G)$ (respectively, $J(C_5)$) denotes the  cover ideal of the graph $G$ (respectively, odd cycle subgraph $C_5)$. 
In addition, we can deduce from  \cite[Proposition 3.6]{NKA}  that 
 $\mathrm{Ass}_S(S/J(C_5)^s)= \mathrm{Ass}_S(S/J(C_5))\cup \{(x_1, x_2, x_3, x_4, x_5)\}$  for all $s\geq 2$.
 We therefore have, for all $s\geq 2$, 
 \begin{align*}
& \mathrm{Ass}_{R}(R/I^s)=\mathrm{Ass}_{R}(R/J(G)^s)=\\
 & \{(x_1,x_2),   (x_2, x_3),   (x_3, x_4),   (x_4, x_5), (x_5, x_1),  (x_1, x_6),  (x_1, x_2, x_3, x_4, x_5)\}.
 \end{align*}
 This implies that $I$  is nearly normally torsion-free. We now turn our attention to show the normality of $I$. It follows from \cite[Theorem 1.10]{ANR} that 
 $J(C_5)$ is normal, and so  $I$ is normal  according to  \cite[Theorem 2.7]{ANKRQ}. In particular,  \cite[Theorem 6.2]{RNA}  yields that  $I$ satisfies  the strong persistence property. Here, our aim is to verify that, for all $t\geq 2$ and $\ell \geq 1$,  $I^t$ cannot be presented as below  
\begin{equation}\label{21}
 I^t =  (x_1,x_2)^t  \cap  (x_2, x_3)^t  \cap  (x_3, x_4)^t \cap  (x_4, x_5)^t \cap (x_5, x_1)^t  \cap (x_1, x_6)^t 
 \cap (x_1, x_2, x_3, x_4, x_5)^{\ell t}. 
\end{equation} 
 On the contrary, assume that $I^t$ can  be presented in  the above  format  for some $t\geq 2$ and $\ell \geq 1$. 
  The evaluation of $x_1, x_4$ at $0$ in both sides of the equality (\ref{21}), we get 
 $(x_2x_3x_5x_6)^tR= (x_2x_3x_5x_6)^tR \cap (x_2, x_3,  x_5)^{\ell t}$, and so $(x_2x_3x_5x_6)^t \in  (x_2, x_3,  x_5)^{\ell t}$. So, there exists 
 a monomial $u=x_2^{a} x_3^{b} x_5^{c} \in \mathcal{G}((x_2, x_3,  x_5)^{\ell t})$ such that $u \mid (x_2x_3x_5x_6)^t$, where $a,b,c \geq 0$ and 
 $a+b+c=\ell t$. From $u \mid (x_2x_3x_5x_6)^t$, we obtain $a\leq t$, $b\leq t$, and $c\leq t$. This  implies that $\ell t=a+b+c\leq 3t$, and hence  
 $\ell \leq 3$. It can be checked  that  
$$x_1^{t-1} x_2 x_3^{t-1} x_4 x_5^{t-1} x_6 \in  \bigcap_{i=1}^4 (x_i, x_{i+1})^t \cap (x_5, x_1)^t \cap (x_1, x_6)^t \cap (x_1, x_2, x_3, x_4, x_5)^{\ell t}.$$ 
We will show that $x_1^{t-1} x_2 x_3^{t-1} x_4 x_5^{t-1} x_6 \notin I^t$. Assume the contrary, then there are nonnegative integers $a_i$, where $i=1, \ldots, 5$,  with 
$a_1+a_2+a_3+a_4+a_5=t$   such that 
$$(x_1x_3x_5)^{a_1} (x_1x_3x_4)^{a_2} (x_1x_2x_4)^{a_3} (x_2x_4x_5x_6)^{a_4} (x_2x_3x_5x_6)^{a_5} \mid x_1^{t-1} x_2 x_3^{t-1} x_4 x_5^{t-1} x_6.$$
The total degree on the right is $3t$, while the total degree of the divisor is $3t+a_4+a_5$, hence, $a_1+a_2+a_3 =t$. 
Comparison of the degree of $x_1$ on each side leads to a contradiction.
 Consequently, $I$ is not of well-nearly normally torsion-free type. 
 }
 \end{example}


We are in a position to state the second main result of this section in the subsequent theorem. 

\begin{theorem}  \label{Normal+Well-NNTF}
 Let $I\subset R=K[x_1, \ldots, x_n]$  be   of  well-nearly  normally torsion-free  type. Then the following statements hold:
\begin{itemize}
\item[(i)]  $I$ is normal.
\item[(ii)]  $I$ has the strong persistence property. 
\item[(iii)]  $I$ has the persistence property. 
\item[(iv)]  $I$ has the symbolic strong persistence property.
\end{itemize}
\end{theorem}

\begin{proof}  
(i)  To show the normality of $I$, one has to prove that $I^s=\overline{I^s}$ for all $s\geq 1$. To do this, fix $s\geq 1$. Since $I^s \subseteq \overline{I^s}$, it is enough for us to verify that $\overline{I^s} \subseteq I^s$. Let $u$ be an arbitrary monomial in $\overline{I^s}$. 
It follows at once  from \cite[Theorem 1.4.2]{HH1} that $u^z \in I^{sz}$ for some $z\geq 1$. Because $I$ is of  well-nearly normally torsion-free type,   this implies that    there exist  some $\ell, m\geq 1$ and some prime monomial ideal $\mathfrak{q}$  such that $\mathrm{Ass}(I^t)=\mathrm{Min}(I)$ for all $1\leq t \leq m$, and  $\mathrm{Ass}(I^t) \subseteq \mathrm{Min}(I) {\cup} \{\mathfrak{q}\}$ for all $t\geq m+1$ such that  if $\mathfrak{q}\in \mathrm{Ass}(I^k)$ for some $k\geq m+1$, then   $I^k=\bigcap_{\mathfrak{p}\in \mathrm{Min}(I)}\mathfrak{p}^k  \cap \mathfrak{q}^{k\ell}$   is a minimal primary decomposition of $I^k$. We can deduce  from    \cite[Proposition 4.3.25]{V1} that  $I^{(s)}= \bigcap_{\mathfrak{p}\in \mathrm{Min}(I)} \mathfrak{p}^{s}$, where $I^{(s)}$ denotes the $s$-th  symbolic power of $I$. We thus get  
$I^{s}=I^{(s)}= \bigcap_{\mathfrak{p}\in \mathrm{Min}(I)} \mathfrak{p}^{s}$
 or  $I^{s}=\bigcap_{\mathfrak{p}\in \mathrm{Min}(I)} \mathfrak{p}^{s} \cap \mathfrak{q}^{s\ell}$.  This implies that   $u^z\in  \mathfrak{p}^{sz}$ for all $\mathfrak{p}\in \mathrm{Min}(I)$, and so $u\in \overline{\mathfrak{p}^s}$ for all $\mathfrak{p}\in \mathrm{Min}(I)$. In addition, it follows from  \cite[Theorem 1.4.6]{HH1} that every prime monomial ideal is  normal. This gives that $u\in \mathfrak{p}^s$ for all $\mathfrak{p}\in \mathrm{Min}(I)$. Hence, $u\in \bigcap_{\mathfrak{p}\in \mathrm{Min}(I)} \mathfrak{p}^{s}$. If $u^z\in  \mathfrak{q}^{sz\ell}$, then $u \in \overline{\mathfrak{q}^{s\ell}}$, and \cite[Theorem 1.4.6]{HH1}  yields that $u\in \mathfrak{q}^{s\ell}$. Therefore, we obtain   $u\in I^s$, and so 
   $I$ is normal, as required. \par 
(ii) Since normality implies SPP, the claim follows readily from (i).\par 
(iii)-(iv)  Because  SPP implies PP and symbolic PP, both claims are derived from (ii).
\end{proof}


\section{An argument on the existence of   embedded associated prime ideals}

In this section, our aim is to give some results which are related to the existence of   embedded associated prime ideals  in the associated primes set of powers of monomial ideals. In particular, we show that these consequences can be employed in studying the edge and cover ideals of cones of graphs. 
We begin by stating the first main result of this section.

\begin{proposition}\label{Pro.Maximal1}
Let  $I \subset R=K[x_1,\ldots, x_n, x_{n+1}]$ be a monomial ideal, where $\mathrm{gcd}(u, x_{n+1})=1$ for all $u\in \mathcal{G}(I)$, and $\mathfrak{m}$
the ideal generated by the variables. Further, let   $h=fg$ with $1\neq f\in R, g\in \mathcal{G}(I)$ and $\mathrm{gcd}(h,x_{n+1})=1$,  and  $L:=x_{n+1}I+hR$.   If  $(x_1, \ldots, x_n)\in \mathrm{Ass}(I^t)$ 
for some  $t\geq 1$,  then   $\mathfrak{m}\in \mathrm{Ass}(L^{t+1})$.
\end{proposition}
\begin{proof}
Let  $(x_{1}, \ldots, x_{n})\in \mathrm{Ass}(I^t)$ for some $t\geq 1$. Hence, there exists some monomial $v$ such that 
 $(I^t:v)=(x_1, \ldots, x_n)$.  Since $x_{n+1}\notin \mathrm{supp}(I)$, we can assume  $\mathrm{gcd}(v,x_{n+1})=1$  according to  \cite[Lemma 2.1]{KHN2}. We claim that 
    $\mathfrak{m}=(L^{t+1} : x_{n+1}^tvh)$. To see this, we consider the following equalities:  
\begin{align*}
(L^{t+1} : x_{n+1}^tvh)=& (\sum_{\alpha=0}^{t+1}  x_{n+1}^{\alpha}I^{\alpha}h^{t+1-\alpha} :  x_{n+1}^tvh)\\
=& (x_{n+1}^{t+1}I^{t+1} :  x_{n+1}^tvh) + (x_{n+1}^{t}I^th : x_{n+1}^tvh) \\
& + \sum_{\alpha=0}^{t-1} (x_{n+1}^{\alpha}I^{\alpha}h^{t+1-\alpha} : x_{n+1}^tvh)\\
= & x_{n+1}(I^{t+1} : vh) + (I^t : v) + \sum_{\alpha=0}^{t-1} (I^{\alpha} h^{t-\alpha} : v).
\end{align*} 
Due to $h\in I$, we get  $(I^{\alpha} h^{t-\alpha} : v) \subseteq (I^t : v)$ for all $\alpha=0, \ldots, t-1$, and so 
$\sum_{\alpha=0}^{t-1} (I^{\alpha} h^ 
{t-\alpha} : v) \subseteq (I^t : v).$ In addition, since $x_i v\in I^t$ for all $i=1, \ldots, n$, $g\in \mathcal{G}(I)$,  and  
 $f\neq 1$, this implies that  $vh\in I^{t+1}$. Consequently, we have $(I^{t+1} : vh)=R$, and so 
$(L^{t+1}:x_{n+1}^tvh)=x_{n+1}R+(x_1, \ldots, x_n)=\mathfrak{m},$ that is,  $\mathfrak{m}\in \mathrm{Ass}(L^{t+1})$, as desired. 
\end{proof}


We are ready to give the second main result.

\begin{proposition}\label{Pro.Maximal2}
Assume that  $R, I,$ and $\mathfrak{m}$ are as in Proposition~\ref{Pro.Maximal1}. Moreover, let $I\ne \mathfrak{q}=(x_1, \ldots, x_n)$ and $L:=I+x_{n+1}\mathfrak{q}$.
  Suppose that   $(I^t:v)=\mathfrak{q}$ for some $t\geq 1$ and a monomial $v\in R$, such that  if  $x_\lambda v=f_1 \cdots f_t M$, where  $1\leq \lambda \leq n$ and  $M\in R, f_1, \ldots, f_t \in \mathcal{G}(I)$ are monomials, then $Mf_j/x_{\lambda}\neq 1$ for some $1\leq j \leq t$.  Then   $\mathfrak{m}\in \mathrm{Ass}(L^{t})$. 
\end{proposition}
\begin{proof}
First note that, based on  \cite[Lemma 2.1]{KHN2},  we can assume  $\mathrm{supp}(v) \subseteq \mathrm{supp}(I)$.  In particular,  from  $x_{n+1}\notin \mathrm{supp}(I)$, we have $\mathrm{gcd}(v,x_{n+1})=1$. It follows from our assumption that there exists some   $1\leq \lambda \leq n$ such that   if $x_\lambda v=f_1 \cdots f_t M$ for some monomials $f_1, \ldots, f_t \in \mathcal{G}(I)$ and monomial $M\in R$, 
 then $Mf_j/x_{\lambda}\neq 1$ for some $1\leq j \leq t$. Since $(I^t:v)=\mathfrak{q}$, this implies that $v\notin I^t$, and so $x_\lambda\nmid M$. Without loss of generality, assume that  $f_1=x_\lambda h$  for some monomial $h\in R$,  and hence we can write   $v=(f_2 \cdots f_t)hM$ with 
   $hM\neq 1$. This gives rise to  $v\in I^{t-1} \mathfrak{q}$. In  light of the binomial expansion theorem, we have the following     equalities  
\begin{align*}
(L^{t} : v)=& (\sum_{\alpha=0}^{t}  x_{n+1}^{t-\alpha}\mathfrak{q}^{t-\alpha}I^{\alpha} : v)\\
= &  (I^t:v) + \sum_{\alpha=0}^{t-2} x_{n+1}^{t-\alpha}(\mathfrak{q}^{t-\alpha}I^{\alpha} : v) + x_{n+1}(I^{t-1}\mathfrak{q}:v)\\
=& \mathfrak{q} + x_{n+1}R = \mathfrak{m}.
\end{align*} 
This implies that  $\mathfrak{m}\in \mathrm{Ass}(L^{t})$, and  the proof is  finished. 
\end{proof}


Here, we  provide  an  example, which shows that  the converse  of  Propositions \ref{Pro.Maximal1}    and  \ref{Pro.Maximal2}, respectively, is   not true in general.

\begin{example}
{\em 
(i) In $R=K[x_1, \ldots, x_8]$, consider  the square-free monomial ideal
\begin{align*}
L:=(& x_2x_5x_6x_8, x_1x_5x_6x_8, x_3x_4x_6x_8, x_1x_4x_6x_8, x_2x_3x_6x_8,  x_3x_4x_5x_8, x_2x_4x_5x_8,  \\
&  x_1x_3x_5x_8, x_1x_2x_4x_8, x_1x_2x_3x_8,  x_1x_2x_3x_4, x_1x_4x_5x_7x_8).
\end{align*}
Now, set $h:=fg$, where $g=x_1x_2x_3$ and $f=x_4$,  and 
\begin{align*}
I:=(& x_2x_5x_6,x_1x_5x_6,x_3x_4x_6,x_1x_4x_6,x_2x_3x_6, x_3x_4x_5,x_2x_4x_5,    x_1x_3x_5,x_1x_2x_4,  \\
&  x_1x_2x_3, x_1x_4x_5x_7).
\end{align*}
It is easy to see  that $L=x_8I+ hR$ satisfies the conditions in Proposition~\ref{Pro.Maximal1}. Using  {\it Macaulay2} \cite{GS}, we can deduce that $$(x_1, x_2, x_3, x_4, x_5, x_6, x_7, x_8)\in \mathrm{Ass}(L^4), \text{~but~}
(x_1, x_2, x_3, x_4, x_5, x_6, x_7)\notin \mathrm{Ass}(I^3).$$  

(ii) Consider the  ideal  $I=(x_1x_2, x_2x_3, x_3x_4, x_4x_5, x_5x_1)\in K[x_1, \ldots, x_6]$ and let  $L= I+x_6(x_1, x_2, x_3, x_4, x_5)$. 
Using  {\it Macaulay2} \cite{GS},  we find that 
$\mathfrak{m}\in \mathrm{Ass}(L^2)$, while $(x_1, x_2, x_3, x_4, x_5)\notin \mathrm{Ass}(I^2)$.
}
 \end{example}
 

In what follows, we present some applications of Propositions \ref{Pro.Maximal1}    and  \ref{Pro.Maximal2}.
 
 \begin{definition} (\cite[Definition 10.5.4]{V1})
  The {\it cone} $C(G)$, over the graph $G$, is obtained by adding a new vertex $w$ to $G$ and joining every vertex of $G$ to $w$. 
  \end{definition} 


 \begin{lemma} \label{Cone-Cover-Edge}
 Let   $G$  be  a  connected graph  and    $H:=C(G)$  be  its cone. Then the following statements hold:
 \begin{itemize}
 \item[(i)] If   $(x_i : i\in V(G))\in \mathrm{Ass}(J(G)^t)$ for some $t\geq 1$,  then  $(x_j : j\in V(H))\in \mathrm{Ass}(J(H)^{t+1})$, 
 where $J(G)$  and $J(H)$ denote the cover ideals of $G$ and $H$, respectively. 
 
 \item[(ii)]  If  $(x_i : i\in V(G))\in \mathrm{Ass}(I(G)^t)$  for some  $t\geq 2$,  then   $(x_j : j\in V(H))\in \mathrm{Ass}(I(H)^{t})$,  
 where $I(G)$   and $I(H)$ denote the edge ideals of $G$ and $H$, respectively. 
 \end{itemize}
  \end{lemma}
 
 \begin{proof}
 (i)  Assume that  the  cone $H=C(G)$  is obtained by adding the new vertex $w$ to $G$ and joining every vertex of $G$ to $w$.
 Then one can easily see that 
 $$J(H)=J(G)\cap (x_w, \prod_{i\in V(G)}x_i)=x_wJ(G)+ \prod_{i\in V(G)}x_iR,$$ 
 where $R=K[x_i : i\in V(H)]$.  Since $G$  is connected,  we get  $\prod_{i\in V(G)}x_i \notin \mathcal{G}(J(G))$. Now,  the  claim is a direct consequence of   Proposition \ref{Pro.Maximal1}. 
 
 (ii) It is routine to check that  $I(H)=I(G)+ x_w(x_i : i\in V(G)).$ Suppose that  $(x_i : i\in V(G))\in \mathrm{Ass}(I(G)^t)$ for some $t\geq 2$, say $(I(G)^t:v)=(x_i : i\in V(G))$ for some monomial $v\in R$. Then there exist monomials $f_1, \ldots, f_t \in \mathcal{G}(I(G))$ and a monomial $M\in R$ such that $x_1v=f_1 \cdots f_t M$. Since $x_1 \nmid M$, without loss of generality, assume that $x_1\mid f_1$. Due to $f_1\in \mathcal{G}(I(G))$ and $G$ being  connected, there exists some positive integer  $1\neq r\in V(G)$ such that $f_1=x_1x_r$, and so $f_1M/x_1=x_rM\neq 1$.  As   $I(G)\neq (x_i : i\in V(G))$, the claim can be deduced immediately from  Proposition \ref{Pro.Maximal2}, and the proof is  finished. 
  \end{proof}


As an immediate consequence of  Lemma \ref{Cone-Cover-Edge}, we have the following  result. 

\begin{corollary}
Let $K_n$ denote the complete graph of order $n$.  Then the following statements hold:
\begin{itemize}
\item[(i)] $(x_1, \ldots, x_n)\in \mathrm{Ass}(J(K_n)^{t})$ for all $t\geq n-1$ and  $n\geq 2$. 

\item[(ii)]  $(x_1, \ldots, x_n)\in \mathrm{Ass}(I(K_n)^{t})$ for all $t\geq 2$ and $n\geq 3$. 
\end{itemize}
\end{corollary}

\begin{proof} 
(i) By virtue of $K_n$  being  the cone of $K_{n-1}$  and  Lemma  \ref{Cone-Cover-Edge}, we can establish  by induction on $n$  that $(x_1, \ldots, x_n)\in \mathrm{Ass}(J(K_n)^{n-1})$. On the other hand, according to \cite[Corollary 1.7]{ANR}, we conclude  that $J(K_n)$ is a normal square-free monomial ideal, and so has the persistence property. This gives rise to  $(x_1, \ldots, x_n)\in \mathrm{Ass}(J(K_n)^{t})$ for all $t\geq n-1$  and  $n\geq 2$, as required. 

(ii) First note that  $(x_1, x_2, x_3) \in \mathrm{Ass}(I(K_3)^2)$. By  an argument similar to part (i), we obtain   
$(x_1, \ldots, x_n)\in \mathrm{Ass}(I(K_n)^{2})$. In addition, \cite[Theorem 7.7.14]{V1} yields that $I(K_n)$  has the persistence property, and consequently 
 $(x_1, \ldots, x_n)\in \mathrm{Ass}(I(K_n)^{t})$ for all $t\geq 2$  and  $n\geq 3$, as claimed.  
\end{proof}
We present  the following example, which is related to  Lemma \ref{Cone-Cover-Edge}(i). 
\begin{example}
{\em 
Let $G=C_{2n-1}$ denote the odd cycle graph of order $2n-1$, $n\geq 2$, with the vertex set $V(C_{2n-1})=\{1, \ldots, 2n-1\}$ and the edge set 
$$E(C_{2n-1})=\{\{i, i+1\} : i=1, \ldots, 2n-2\} \cup \{2n-1, 1\}.$$ Then the cone $H=C(C_{2n-1})$ over the graph $C_{2n-1}$ can be obtained by adding the new vertex $2n$ and joining every vertex of $C_{2n-1}$ to $2n$. It is obvious that  $H=W_{2n}$ is the wheel graph of order $2n$. 
On the other hand, we know from \cite[Propositions 3.6 and 3.10]{NKA} that $(x_1, \ldots, x_{2n-1})\in \mathrm{Ass}(J(C_{2n-1})^s)$ and 
$(x_1, \ldots, x_{2n-1}, x_{2n}) \in  \mathrm{Ass}(J(W_{2n})^{s+1}),$ where $s\geq 2$. 
}
\end{example}


Can we inquire if every nearly normally torsion-free monomial ideal is a normal ideal in general? The answer is no. 
   For example, consider the monomial ideal   $I=(x^5,x^4y, xy^4, y^5z)$  in $R=K[x,y,z]$. According to \cite[Example 2.8]{Nasernejad1}, we have 
$\mathrm{Min}(I)=\{(x,y), (x,z)\}$ and $\mathrm{Ass}(R/I^s)=\mathrm{Min}(I) \cup \mathrm\{(x,y,z)\}$ for all $s\geq 2$. This means that 
$I$ is  nearly normally torsion-free. Now, set $u:=x^4y^6$. It is easy to check that   
$u^2=x^{8}y^{12}=(x^5)(xy^4)(xy^4)(xy^4)\in I^4$, while $u\notin I^2$. Hence, $I^2$ is not integrally closed, and so $I$ is not normal.


Suppose that $I$ is   a normally torsion-free square-free  monomial ideal with $\mathcal{G}(I)=\{u_1, \ldots, u_m\}$. Assume that  we delete one of the members of $\mathcal{G}(I)$, say $u_m$, and add a new square-free monomial, say $v$. There is no guarantee that  the new square-free monomial
 ideal $(u_1, \ldots, u_{m-1})+(v)$ is still normally torsion-free.  
 To see a counterexample, let $C_9$ denote the odd cycle graph  of order $9$ with the   edge set 
$E(C_9)=\{\{i, i+1\}~:~i=1, \ldots, 8\}\cup \{\{9,1\}\}.$  
 Assume that  $L:=DI(C_9)$ is  the dominating ideal of $C_9$ in the polynomial ring $R=K[x_1, \ldots, x_9]$. Based on
    \cite[Example 2.2]{Nasernejad1}, $L$ is normally torsion-free. In particular, $x_3x_6x_9 \in \mathcal{G}(L)$. 
   Now, let  $J$ be  the square-free  monomial ideal with 
     $\mathcal{G}(J)=(\mathcal{G}(L)\setminus \{x_3x_6x_9\})\cup \{x_3x_6x_9x_5\}$. Using {\it Macaulay2} \cite{GS}, we obtain  $\#\mathrm{Ass}(J)=13$ while $\#\mathrm{Ass}(J^2)=18$. This means that $J$ is not normally torsion-free.

 In the following result, we show that under certain conditions the square-free monomial ideal $(u_1, \ldots, u_{m-1})+(v)$  is normally torsion-free as well.

\begin{theorem} \label{NTF3}
Let  $I$ be  a normally torsion-free square-free  monomial ideal in a polynomial ring $R=K[x_1, \ldots, x_n]$ over a field $K$ with 
$\mathcal{G}(I)=\{u_1, \ldots, u_m\}$. Let   $L:=(u_1, \ldots, u_{m-1})S + fx_{n+1}\cdots x_{n+t}S$, where $S=R[x_{n+1},\ldots, x_{n+t}]$, 
$f$ a square-free monomial  with  $f \mid u_m$,  $fx_{n+1}\cdots x_{n+t}\notin \mathfrak{p}^2$ for any $\mathfrak{p}\in \mathrm{Min}(L)$, and     $(u_1, \ldots, u_{m-1})S$  is normally torsion-free.  Then the following statements hold:
\begin{itemize}
\item[(i)] $L$  is normally torsion-free. 
\item[(ii)] $L$  is nearly normally torsion-free. 
\item[(iii)]  $L$ is normal.
\item[(iv)]  $L$ has the strong persistence property. 
\item[(v)]  $L$ has the persistence property. 
\item[(vi)]  $L$ has the symbolic strong persistence property.
\end{itemize}
\end{theorem}  
\begin{proof}
(i)  To simplify the notation, put  $v:= fx_{n+1}\cdots x_{n+t}$.  Suppose, on  the contrary,  that  $L$  is not normally torsion-free. Let  $\ell$ be minimal such that $L^\ell$ has embedded prime ideals.  Due to $L$ is a square-free monomial ideal, this gives that $\mathrm{Ass}(L)=\mathrm{Min}(L)$, and so  $\ell \geq 2$.  Pick  an arbitrary  element $\mathfrak{p}\in \mathrm{Min}(L)$. We want to show that $(L^{\ell}:v)=L^{\ell-1}$. It follows immediately from $v\in L$ that  $L^{\ell-1} \subseteq (L^\ell:v)$. Hence,  it is sufficient for us to verify  the reverse inclusion. To do this, take a  monomial $h$ in $(L^\ell:v)$. We thus get   $hv\in L^\ell$, and hence  there exist  monomials $g_1, \ldots, g_\ell \in \mathcal{G}(L)$ and some monomial $f'$ in $S$ such that $hv=g_1 \cdots g_\ell f'$.  One can conclude readily  from   $\mathfrak{p}\in \mathrm{Min}(L)$  that $hv\in \mathfrak{p}^\ell$. In addition, we deduce from the assumption    $v\notin  \mathfrak{p}^2$  that  $\mathfrak{p}$ contains exactly one variable that divides $v$. Consequently,  $h\in \mathfrak{p}^{\ell-1}$. Since   $\mathfrak{p}$ is arbitrary,  we get   $h\in \bigcap_{\mathfrak{p}\in \mathrm{Min}(L)}\mathfrak{p}^{\ell-1}$. It follows  from   \cite[Proposition 4.3.25]{V1} that  $h\in L^{(\ell-1)}$.  Thanks to  $\ell$ being  minimal such that $L^\ell$ has embedded prime ideals, this leads to  $L^{\ell-1}$ has no embedded prime ideals, and so $L^{(\ell-1)}=L^{\ell-1}$. We therefore have  $h\in L^{\ell-1}$. Accordingly, one has  $(L^{\ell}:v)\subseteq L^{\ell-1}$, and so   $(L^{\ell}:v)=L^{\ell-1}$. Suppose that  $\mathfrak{q}$ is  an embedded prime ideal of $L^\ell$. 
 This implies that  $\mathfrak{q}=(L^\ell:z)$ for some monomial $z\in S$.    We claim that  $\mathfrak{q}\setminus x_i\notin \mathrm{Ass}((L\setminus x_i)^\ell)$ for any $x_i\in \mathrm{supp}(f)$.  On the contrary, assume that  $\mathfrak{q}\setminus x_c\in \mathrm{Ass}((L\setminus x_c)^\ell)$ for some  $x_c\in \mathrm{supp}(f)$. Since $I$ is normally torsion-free,  we can derive from \cite[Theorem 3.21]{SN} that $I\setminus x_c$ is normally torsion-free. 
  Observe that the hypothesis  $\mathrm{supp}(f)\subseteq \mathrm{supp}(u_m)$  yields  that   $I\setminus x_i=L\setminus x_i$ for any $x_i\in \mathrm{supp}(f)$. This gives that  $I\setminus x_c=L\setminus x_c$,  and so $L\setminus x_c$ is normally torsion-free. We thus get  $\mathfrak{q}\setminus x_c\in \mathrm{Min}(L\setminus x_c)$, and hence $\mathfrak{q}\in \mathrm{Min}(L)$. 
 This leads to a contradiction. Consequently, one has $\mathfrak{q}\setminus x_i\notin \mathrm{Ass}((L\setminus x_i)^\ell)$ for any $x_i\in \mathrm{supp}(f)$. In view of \cite[Corollary 4.5]{SNQ}, we get  $x_i\mid z$ for any $x_i\in \mathrm{supp}(f)$, that is,  $f\mid z$.  Here, our aim is to demonstrate that $\mathfrak{q}\setminus x_\lambda\notin \mathrm{Ass}((L\setminus x_\lambda)^\ell)$ for any  $n+1\leq \lambda \leq n+t$. 
Assume, for the  sake of contradiction, that    $\mathfrak{q}\setminus x_\lambda\in \mathrm{Ass}((L\setminus x_\lambda)^\ell)$ for some $n+1\leq \lambda \leq n+t$. Since  $L\setminus x_\lambda=(u_1, \ldots, u_{m-1})S$  and $(u_1, \ldots, u_{m-1})S$ is normally torsion-free,  we have 
$L\setminus x_\lambda$ is normally torsion-free, and so $\mathfrak{q}\setminus x_\lambda\in \mathrm{Min}(L\setminus x_\lambda)$.  
This yields that  $\mathfrak{q}\in \mathrm{Min}(L)$, which is a contradiction.  
  Accordingly, we have $\mathfrak{q}\setminus x_\lambda\notin \mathrm{Ass}((L\setminus x_\lambda)^\ell)$ for any  $\lambda=n+1, \ldots,  n+t$. It follows again  from    \cite[Corollary 4.5]{SNQ} that $x_\lambda\mid z$ for any  $\lambda=n+1, \ldots,  n+t$, that is, $x_{n+1} \cdots x_{n+t} \mid z$.   This gives rise to  $z=vM$ for some monomial $M$. Hence,  $\mathfrak{q}=(L^\ell:z)=(L^{\ell-1}:M)$, and so   $\mathfrak{q}\in \mathrm{Ass}(L^{\ell-1})$,  which  contradicts  the fact that $L^{\ell-1}$ has no embedded prime ideals. This implies that  $L$ is normally torsion-free, as desired.  \par
 (ii) It is obvious from  the definition of nearly normally torsion-freeness and (i). \par 
(iii) Follows from (i).  \par 
Statements  (iv)-(vi)  are shown exactly as in  the  proof of  Theorem \ref{Normal+Well-NNTF}.    
  \end{proof}


The following   is an immediate consequence of  Theorem \ref{NTF3}.

\begin{corollary} \label{NTF4}
Let  $I$ be  a normally torsion-free square-free  monomial ideal in  $R=K[x_1, \ldots, x_n]$  with 
$\mathcal{G}(I)=\{u_1, \ldots, u_m\}$. Let   $L:=(u_1, \ldots, u_{m-1},gu_m)S$, where  $S=R[x_{n+1},\ldots, x_{n+t}]$, 
 $g$ be a square-free monomial in $S$ with  $\mathrm{supp}(g) \subseteq \{x_{n+1}, \ldots, x_{n+t}\}$,  $gu_m \notin \mathfrak{p}^2$  for any $\mathfrak{p}\in \mathrm{Min}(L)$, and     $(u_1, \ldots, u_{m-1})S$ be  normally torsion-free. Then the following statements hold:
\begin{itemize}
\item[(i)] $L$  is normally torsion-free. 
\item[(ii)] $L$  is nearly normally torsion-free. 
\item[(iii)]  $L$ is normal.
\item[(iv)]  $L$ has the strong persistence property. 
\item[(v)]  $L$ has the persistence property. 
\item[(vi)]  $L$ has the symbolic strong persistence property.
\end{itemize}
\end{corollary}


We close this section with a pair of  open questions.

\begin{question}  \label{Open-Question-1}
Let $I \subset R=K[x_1, \ldots, x_n]$ be a  square-free monomial ideal  and   $\mathfrak{q}$ be  a prime monomial ideal in $R$. Let  $L:=I \cap (\mathfrak{q}, x_r)$  such that   $x_r\notin \mathrm{supp}(\mathfrak{q}) \cup \mathrm{supp}(I)$ and  $\mathrm{Ass}(L)=\mathrm{Ass}(I) \cup \{(\mathfrak{q}, x_r)\}$.   If    $\mathfrak{m}=(x_1, \ldots, x_n)\in\mathrm{Ass}(L^s)$ for some $s\geq 2$, then can we deduce that   $\mathfrak{m}\setminus \{x_r\} \in  \mathrm{Ass}(I^s)$ or $\mathfrak{m}\setminus \{x_r\} \in  \mathrm{Ass}((I\cap \mathfrak{q})^s)$?  
\end{question}

\begin{question}  \label{Open-Question-2}
Let $I \subset R$ be a  square-free monomial ideal  and   $\mathfrak{q}$ be  a prime monomial ideal in $R$. Let  $L:=I \cap (\mathfrak{q}, x_r)$  such that   $x_r\notin \mathrm{supp}(\mathfrak{q}) \cup \mathrm{supp}(I)$ and  $\mathrm{Ass}(L)=\mathrm{Ass}(I) \cup \{(\mathfrak{q}, x_r)\}$.   If    $\mathfrak{p}=(L^s:x_r^{\theta}v)$ for some $s\geq 2$,   $\theta \geq 0$, and monomial $v$  with 
$x_r\nmid v$, then  can we deduce that $\mathfrak{p}\setminus  x_r \in  \mathrm{Ass}(I^{\theta}(I\cap \mathfrak{q})^{s-\theta})$?
\end{question}


\begin{remark}{\em 
(i) With the notation from   Question \ref{Open-Question-1}, it is possible  that we simultaneously have, for some $s\geq 2$,  $\mathfrak{m}=(x_1, \ldots, x_n)\in\mathrm{Ass}(L^s)$ and  $\mathfrak{m}\setminus \{x_r\} \in \mathrm{Ass}(I^{\theta}(I\cap \mathfrak{q})^{s-\theta})$  for all
  $0\leq \theta \leq s$. To see an example, consider the following square-free monomial ideal in $R=K[x_1, x_2, x_3, x_4, x_5, x_6]$, 
$$L=(x_3x_4,x_2x_4,x_3x_5x_6,x_1x_5x_6,x_1x_3x_6,x_1x_4x_5,x_2x_3x_5,x_1x_2x_5,x_1x_2x_3).$$
Using {\it Macaulay2} \cite{GS}, we obtain 
$$\mathrm{Ass}(L)=\{(x_1,x_2,x_3), (x_1,x_3,x_4), (x_1, x_4,x_5), (x_2,x_3,x_5), (x_3,x_4,x_5), (x_2,x_4,x_6)\}.$$
Let $\mathfrak{q}:=(x_2,x_4)$, $s=3$, $x_r=x_6$, $\mathfrak{m}=(x_1, x_2, x_3, x_4, x_5, x_6)$, and 
$$\mathrm{Ass}(I)=\{(x_1,x_2,x_3), (x_1,x_3,x_4), (x_1, x_4,x_5), (x_2,x_3,x_5), (x_3,x_4,x_5)\}.$$
 As  $x_1x_3\in I \setminus (I\cap \mathfrak{q})$, we get $I\cap \mathfrak{q}\neq I$. Using {\it Macaulay2} \cite{GS}, we deduce that   $\mathfrak{m}\in\mathrm{Ass}(L^3)$,    $\mathfrak{m}\setminus \{x_6\} \in  \mathrm{Ass}((I\cap \mathfrak{q})^3)$, 
  $\mathfrak{m}\setminus \{x_6\} \in  \mathrm{Ass}(I^2(I\cap \mathfrak{q}))$, $\mathfrak{m}\setminus \{x_6\} \in  \mathrm{Ass}(I(I\cap \mathfrak{q})^2)$, and   $\mathfrak{m}\setminus \{x_6\} \in  \mathrm{Ass}(I^3)$. 

(ii) With the notation from  Question \ref{Open-Question-1}, it is possible  that we  have, for some $s\geq 2$,  $\mathfrak{m}=(x_1, \ldots, x_n)\in\mathrm{Ass}(L^s)$ and $\mathfrak{m}\setminus \{x_r\} \notin \mathrm{Ass}(I^{\theta}(I\cap \mathfrak{q})^{s-\theta})$  for some  $0\leq \theta \leq s$.
  For instance, in  $R=K[x_1, \ldots, x_6]$,  let  
\begin{align*}
L=&(x_3x_6,x_2x_5,x_3x_4,x_2x_4,x_1x_4,x_2x_3,x_1x_3)\\
 =& (x_1, x_2, x_3) \cap  (x_2, x_3, x_4) \cap  (x_3, x_4, x_5) \cap  (x_1, x_2, x_4, x_6). 
\end{align*}
Assume that  $\mathfrak{q}:=(x_1, x_2, x_4)$, $s=2$, $x_r=x_6$, and  $\mathfrak{m}=(x_1, x_2, x_3, x_4, x_5, x_6)$, and 
$$\mathrm{Ass}(I)=\{(x_1, x_2, x_3), (x_2, x_3, x_4), (x_3, x_4, x_5)\}.$$
Using {\it Macaulay2} \cite{GS}, we  can  conclude  that  $\mathfrak{m}\in\mathrm{Ass}(L^2)$, 
$\mathfrak{m}\setminus \{x_6\} \notin  \mathrm{Ass}(I(I\cap \mathfrak{q}))$, and  $\mathfrak{m}\setminus \{x_6\} \in  \mathrm{Ass}((I\cap \mathfrak{q})^2)$, but    $\mathfrak{m}\setminus \{x_6\} \notin  \mathrm{Ass}(I^2)$. }
\end{remark}


\section{A comparison between   clutters and complement clutters}

Our aim in this section is to compare some algebraic properties of the edge ideals of  clutters and complement clutters. In particular, our motivation comes back  to Question 6.3 in \cite{NKRT}. Firstly,  we need to recall the notion of a clutter. A  {\it clutter} (or {\it simple hypergraph}) $\mathcal{C}$  with vertex set 
$X=\{x_1, \ldots, x_n\}$ is a family  of subsets of $X$, called {\it edges}, none of which is included in another. The edge ideal 
of a clutter, $I(\mathcal{C})$, is defined similarly as for a graph and is generated by the products of all the vertices in an edge, for every edge of 
$\mathcal{C}$. The cover ideal of a clutter can be defined as the ideal of all monomials $M$ such that given any edge $e$  of $\mathcal{C}$  there is some variable $x_i$ such that $x_i\in e$ and $x_i|M$.  Note that the vertices of these clutters become the variables of the ring to  which the edge ideal and  the  cover ideal are allocated. 
 Given a clutter $\mathcal{C}$ on $\{x_1, \ldots, x_n\}$ with edges $e_1, \ldots, e_r$, we define the {\it complement clutter}, denoted by  $\mathcal{C}^c$, as the clutter whose edges are 
$\{x_1, \ldots, x_n\}\setminus e_i$ for each $i=1, \ldots, r$. In particular, we have  $(\mathcal{C}^c)^c=\mathcal{C}$.

\begin{question} (\cite[Question 6.3]{NKRT}) \label{Q.6.3}
Does $I(\mathcal{C})$ have  the strong persistence property if and only if 
$I(\mathcal{C}^c)$ have  the strong persistence property?
\end{question}

The authors give an example of a clutter  $\mathcal{C}$, where neither $I(\mathcal{C})$ nor $I(\mathcal{C}^c)$   possesses SPP. Inspired by that, we formulate the following.

\begin{question} \label{Complement-Clutter}
\begin{itemize}
\item[(i)] Does $I(\mathcal{C})$ have  the  persistence property if and only if  $I(\mathcal{C}^c)$  has  the  persistence property?
\item[(ii)] Is $I(\mathcal{C})$   normal  if and only if  $I(\mathcal{C}^c)$ is normal?
\item[(iii)] Is  $I(\mathcal{C})$   normally torsion-free  if and only if  $I(\mathcal{C}^c)$ is normally torsion-free?
\end{itemize}
\end{question}

We will provide an example to give a negative answer to the first two questions, which in turn will be a new counterexample to Question \ref{Q.6.3}.

\textbf{Answer.}
(i)-(iii)  We construct an example showing that there is a clutter $\mathcal{C}$ such that $I(\mathcal{C})$ is normal (hence, has SPP and PP) but for which
$I(\mathcal{C}^c)$  does not have PP (hence, does not have SPP and is not normal).
Consider the  clutter
 \begin{align*}
\mathcal{C}:=\{&\{x_2, x_3, x_6\}, \{x_1, x_4, x_5, x_7\}, \{x_4, x_5, x_6, x_7\},\{x_1, x_3, x_6, x_7\},  \{x_3, x_5, x_6, x_7\},   \\ 
& \{x_1, x_3, x_4, x_7\}, \{x_1, x_2, x_6, x_7\}, \{x_1, x_2, x_5 , x_7\},   \{x_2, x_4, x_6, x_7\}, \{x_2, x_3, x_5, x_7\},\\
& \{x_2, x_3, x_4, x_7\}\}. 
\end{align*}
Its edge ideal is $I(\mathcal{C})=( x_2 x_3 x_6, x_1 x_4 x_5 x_7, x_4 x_5 x_6 x_7, x_1 x_3 x_6 x_7, x_3 x_5 x_6 x_7,  x_1 x_3 x_4 x_7, \allowbreak x_1 x_2 x_6 x_7,  x_1 x_2 x_5 x_7, x_2 x_4 x_6 x_7, x_2 x_3 x_5 x_7, x_2 x_3 x_4 x_7). $
 On the other hand, the edge ideal of the complement clutter $\mathcal C^c$ is given by 
 $I(\mathcal{C}^c)=(   x_1 x_4  x_5 x_7,  x_2 x_3 x_6, x_1 x_2 x_3,  x_2 x_4 x_5, \allowbreak x_1 x_2 x_4, x_2 x_5 x_6,   x_3 x_4 x_5, x_3 x_4 x_6,  x_1 x_3 x_5,      x_1 x_4 x_6,   x_1 x_5 x_6). $   

 Our first aim is to show that  $I(\mathcal{C})$ is normal. To do this, let    
$A:=x_3x_4(x_1,x_2)$ and $B:=(x_1x_2, x_1x_4, x_2x_3)$. Since $(x_1, x_2)$ is normal, we  can deduce from \cite[Remark 1.2]{ANR} that 
$A$ is normal. Now, let $P=(V(P), E(P))$ be the path graph with $V(P)=\{x_1, x_2, x_3, x_4\}$ and $E(P)=\{\{x_4, x_1\}, \{x_1, x_2\}, \{x_2, x_3\}\}$. 
Then $B$ is the edge ideal of $P$. Due to $P$  being  a bipartite graph, it follows 
 that $B$ is  normal. In addition, we have $A+B=B$, and so $A+B$ is normal. From  \cite[Theorem 3.1]{NQBM},  we  can derive that
 $C:=A+x_5B=(x_1x_3x_4, x_2x_3x_4, x_1x_2x_5, x_1x_4x_5, x_2x_3x_5)$ is normal. 
Let $D:=(x_4x_5, x_1x_3, x_3x_5, x_2x_4, x_1x_2)$ be the edge  ideal of     $G=(V(G), E(G))$,   the odd cycle  graph with the vertex set $V(G)=\{x_1, x_2, x_3, x_4,x_5\}$ and the  edge set $$E(G)=\{\{x_1, x_3\}, \{x_3, x_5\}, \{x_5, x_4\}, \{x_4, x_2\}, \{x_2, x_1\}\}.$$
On account of $D$ being  the edge ideal of $G$, we conclude from  \cite[Corollary 10.5.9]{V1} that $D$ is normal. Moreover, $C+D=D$, and hence $C+D$ is normal. It follows again from  \cite[Theorem 3.1]{NQBM} that 
\begin{align*}
F:=C+x_6D=(&x_1x_3x_4, x_2x_3x_4, x_1x_2x_5, x_1x_4x_5, x_2x_3x_5, x_4x_5x_6, x_1x_3x_6,\\
& x_3x_5x_6,  x_2x_4x_6, x_1x_2x_6),
\end{align*}
 is normal.  It is straightforward  to check that $I(\mathcal{C})=(x_2x_3x_6)+x_7F$.  Since $(x_2x_3x_6)$ is normal, it is enough to prove that $L:=(x_2x_3x_6)+F$ is normal.  To see this, let   $H=(V(H), E(H))$ be  the  graph with the vertex set $V(H)=\{x_1, x_2, x_3, x_4,x_5\}$ and the  edge set $E(H)=\{\{x_2,x_3\}, \{x_1, x_3\}, \{x_3, x_5\}, \{x_5, x_4\}, \{x_4, x_2\}, \{x_2, x_1\}\}.$ 
 Then it   follows from  \cite[Corollary 10.5.9]{V1} that $M:=I(H)$ is normal. Moreover, note that $L=C+x_6M$. In  light of $C+M=M$, we  obtain  $C+M$  to be  normal.  We  can now deduce from  \cite[Theorem 3.1]{NQBM} that $L=C+x_6M$ is normal.  We therefore get  $I(\mathcal{C})$
  to be normal and, hence, to have SPP and PP.

We now turn our attention to explore  $I(\mathcal{C}^c)$.  Using  {\it Macaulay2} \cite{GS}, we find that $$\mathfrak{m}=(x_1, x_2, x_3, x_4, x_5, x_6, x_7)\in \mathrm{Ass}(I(\mathcal{C}^c)^2) \setminus \mathrm{Ass}(I(\mathcal{C}^c)^3).$$  This shows that $I(\mathcal{C}^c)$ does not satisfy 
 PP, and, therefore, neither SPP nor  normality, as claimed. Hence, the answer is negative to each one of (i)-(iii). 

(iv)  In the following we  give a counterexample.  Consider the  clutter 
$$\mathcal{C}=\{\{x_1, x_2\}, \{x_1, x_3, x_6\}, \{x_4, x_5, x_6\}\},$$  and its  edge ideal  
$I(\mathcal{C})=(x_1x_2, x_1x_3x_6, x_4x_5x_6)$. In addition,  we have 
$I(\mathcal{C}^c)=(x_3x_4x_5x_6, x_1x_2x_3, x_2x_4x_5).$   

Our strategy  to show the normally torsion-freeness of $I(\mathcal{C})$ is to employ 
Theorem  3.7  in \cite{SNQ}. For this purpose, we first show that $x_1x_2\in \mathfrak{p}\setminus \mathfrak{p}^2$ for all 
$\mathfrak{p}\in \mathrm{Min}(I(\mathcal{C}))$, which is straightforward to check, since 
$$\mathrm{Ass}(I(\mathcal{C}))=\mathrm{Min}(I(\mathcal{C}))=\{(x_1,x_4), (x_1,x_5), (x_1, x_6), (x_2, x_6), (x_2, x_3, x_4), (x_2, x_3, x_5)\}.$$ 
 Next, it  is obvious that $I(\mathcal{C})\setminus x_1=(x_4x_5x_6)$ is normally torsion-free. Furthermore, \cite[Lemmas 2.5 and  3.12]{SN} implies  that $I(\mathcal{C})\setminus x_2=(x_1x_3x_6, x_4x_5x_6)=x_6(x_1x_3, x_4x_5)$ is normally torsion-free. 
 Consequently, since $I(\mathcal{C})\setminus x_1$ and $I(\mathcal{C})\setminus x_2$ are normally torsion-free, we deduce from  
 \cite[Theorem 3.7]{SNQ}  that  $I(\mathcal{C})$ is  normally torsion-free. 
  
  We now deal with  the clutter $I(\mathcal{C}^c)$. Using  {\it Macaulay2} \cite{GS}, we detect that 
   $$(x_2, x_3, x_4)\in \mathrm{Ass}(I(\mathcal{C}^c)^2)\setminus  \mathrm{Ass}(I(\mathcal{C}^c)).$$
   This means that $I(\mathcal{C}^c)$ is not normally torsion-free,  which  finishes our argument.


\section{Normally torsion-freeness and normality under polarization}

We continue our argument  with some  questions  which are  devoted to the  normally torsion-freeness and normality of monomial ideals under the polarization operator.  In the sequel, NTF will denote normally torsion-free.

\begin{definition}
 The process of  {\it  polarization}   replaces a power $x_i^t$ by a product of $t$ new variables $x_{(i,1)}\cdots x_{(i,t)}$. We call  
$x_{(i,j)}$ a  {\it shadow} of $x_i$. We will use $\widetilde{I^t}$ to denote the polarization of $I^t$, will use $S_t$ for the new polynomial ring in this polarization, and  will use $\widetilde{w} $ to denote the polarization in $S_t$ of a monomial $w$ in $R=K[x_1, \ldots, x_n]$. The {\it depolarization} of an ideal in $S_t$ is formed by setting $x_{(i,j)}=x_i$ for all $i,j$.
\end{definition}

In general, the polarization operator is more complicated in comparison with other monomial operators, and we pose the following questions with respect to it.

\begin{question} \label{POLARIZATION-NTF}
Let $I$ be a non-square-free monomial ideal in the polynomial ring $K[x_{1},\ldots ,x_{n}]$. Then can we deduce that 
\begin{itemize}
\item[(i)]   if  $\widetilde{I}$ has  the strong persistence property, then  $I$ has the  strong persistence property?
\item[(ii)]  if  $\widetilde{I}$ has  the  persistence property, then $I$ has the   persistence property?
\item[(iii)]  if  $\widetilde{I}$ is normal, then $I$ is normal? 
\item[(iv)] $I$ is   normally torsion-free if and only if   $\widetilde{I}$ is normally torsion-free? (\cite[Question 3.23]{SN})
\end{itemize}
\end{question}

Before we continue, there is a need to define the expansion operator from \cite{BH}, which will be employed as a criterion further on. \par

Let $R=K[x_1, \ldots , x_n]$. For a fixed ordered $n$-tuple $(i_1, \ldots , i_n)$ of positive integers, we consider the  polynomial ring
$$R^{(i_1,\ldots ,i_n)} =K[x_{11}, \ldots , x_{1i_1} , x_{21}, \ldots , x_{2i_2} , \ldots , x_{n1}, \ldots , x_{ni_n}].$$
Let $\mathfrak{p}_j$ be the monomial prime ideal $(x_{j1}, x_{j2}, \ldots , x_{ji_j}) \subseteq R^{(i_1,\ldots,i_n)}$  for all $j=1,\ldots,n$.  
Attached to each monomial ideal $I\subset R$ with a set of monomial generators $\{{\bold x}^{{\bold a}_1} , \ldots , {\bold x}^{{\bold a}_m}\}$, where ${\bold x}^{\bold a_i}={x_1}^{a_i(1)}\cdots {x_n}^{a_i(n)}$  and  $a_i(j)$ denotes the $j$-th component of the vector ${\bold a}_i=(a_i(1),\ldots,a_i(n))$ for all $i=1,\ldots,m$. We define the {\it expansion of I with respect to the n-tuple $(i_1, \ldots, i_n)$}, denoted by $I^{(i_1,\ldots,i_n)}$, to be the monomial ideal $$I^{(i_1,\ldots,i_n)} = \sum_{i=1}^m \mathfrak{p}_1^{a_i(1)}\cdots \mathfrak{p}_n^{a_i(n)}\subseteq R^{(i_1,\ldots,i_n)}.$$
 We simply write $R^*$ and $I^*$,
respectively, rather than $R^{(i_1,\ldots,i_n)}$ and $I^{(i_1,\ldots,i_n)}$.\par
For example, consider $R = K[x_1, x_2, x_3]$ and the ordered $3$-tuple $(1, 3, 2)$. Then
we have $\mathfrak{p}_1 = (x_{11})$, $\mathfrak{p}_2 = (x_{21}, x_{22}, x_{23})$, and $\mathfrak{p}_3 = (x_{31}, x_{32})$. So for the monomial ideal $I = (x_1x_2, x^2_3)$, the ideal $I^* \subseteq K[x_{11}, x_{21}, x_{22}, x_{23}, x_{31}, x_{32}]$ is $\mathfrak{p}_1\mathfrak{p}_2+\mathfrak{p}^2_3$, namely, 
$$I^* = (x_{11}x_{21}, x_{11}x_{22}, x_{11}x_{23}, x^2_{31}, x_{31}x_{32}, x^2_{32}).$$

\textbf{Answer.} (i)-(iii) Consider the  non-square-free monomial  ideal  
$$I=(x_1x_2^2x_3, x_2x_3^2x_4,  x_3x_4^2x_5, x_4x_5^2x_1, x_5x_1^2x_2) \subset R=K[x_1, \ldots, x_5].$$
  By virtue of  $(I^2:I)\neq I$ and   $\mathfrak{m}=(x_1,x_2,x_3,x_4,x_5)\in \mathrm{Ass}_R(R/I)\setminus \mathrm{Ass}_R(R/I^2)$, we deduce that $I$ does not satisfy  PP and, hence, neither SPP nor  normality.
   On the other hand, it has already been proved in \cite[Example 6.8]{RNA} that $\widetilde{I}$ is normal, and, hence, has SPP and PP.

  (iv) Consider the monomial ideal  $I=(x_1,x_2)\cap (x_1,x_2,x_3)^2 \cap (x_1,x_2,x_4)^4$ in $R=K[x_1, x_2, x_3, x_4]$. 
 We  will show that $I$ is normally torsion-free, while $\widetilde{I}$ is not.  
Consider $S = K[x_1, x_3, x_4]$ and the ordered $3$-tuple $(2,1,1)$. Then by considering   $\mathfrak{p}_1 = (x_1,x_2)$, $\mathfrak{p}_2 = (x_3)$, and
 $\mathfrak{p}_3 = (x_4)$,  for the monomial ideal  $F=(x_1)\cap (x_1,x_3)^2 \cap (x_1,x_4)^4\subset S$, we detect that  the ideal 
 $F^* \subset  R$ is $\mathfrak{p}_1\cap (\mathfrak{p}_1,\mathfrak{p}_2)^2\cap (\mathfrak{p}_1, \mathfrak{p}_3)^4$, namely,
$$F^*=(x_1,x_2)\cap (x_1,x_2,x_3)^2 \cap (x_1,x_2,x_4)^4=I.$$
  On account of \cite[Lemma 3.5]{N3}, we know that  $F^*$ is NTF  if and only if $F$ is NTF. Hence,  we will prove the latter.  
To do this, set  $J:=(x^3_1, x^2_1x_4, x_1x^2_4, x_3x^3_4)$. Then it is easy to see that $F=x_1J$. In  light of \cite[Lemma 3.12]{SN}, it is enough to show that  $J$ is NTF. Note that  $J=(x_1,x_3)\cap (x_1,x_4)^3$. 
  In what follows, we show that $J^s=(x_1,x_3)^s\cap (x_1,x_4)^{3s}$ for all $s\geq 1$.  To see this, fix $s\geq 1$. 
It is straightforward to check  that $J^s \subseteq (x_1,x_3)^s\cap (x_1,x_4)^{3s}$.  Hence, one has to prove the reverse inclusion. Take a monomial 
$f\in  \mathcal{G}((x_1,x_3)^s\cap (x_1,x_4)^{3s})$. According to \cite[Proposition 1.2.1]{HH1}, we get $f=\mathrm{lcm}(u,v)$ for some 
$u\in \mathcal{G}((x_1,x_3)^s)$ and $v\in \mathcal{G}((x_1,x_4)^{3s})$. Hence, there exist  integers $0\leq \lambda \leq s$ and $0 \leq \theta \leq 3s$ such that $u=x_1^{\lambda}x_3^{s-\lambda}$ and $v=x_1^{\theta}x_4^{3s-\theta}$. 
This implies that $f=x_1^{\mathrm{max}\{\lambda, \theta\}}x_3^{s-\lambda}x_4^{3s-\theta}$.
 Here, one may consider the following cases.

\textbf{Case 1.}  $0\leq \theta \leq \lambda$. Then $\max\{\lambda, \theta\}=\lambda$, and it  is easy to see that 
$$( x_1x^2_4)^{\lambda}(x_3x^3_4)^{s-\lambda} \mid f.$$

\textbf{Case 2.} $\lambda \leq \theta \leq 2\lambda$. Then $\max\{\lambda, \theta\}=\theta$, which  implies that 
$$(x^2_1x_4)^{\theta - \lambda} (x_1x^2_4)^{2\lambda - \theta} (x_3x^3_4)^{s-\lambda} \mid f.$$

\textbf{Case 3.}   $2\lambda \leq \theta \leq 3\lambda$.  Then $\max\{\lambda, \theta\}=\theta$, and we  can deduce that 
$$(x^3_1)^{\theta - 2\lambda}  (x^2_1x_4)^{3\lambda - \theta}(x_3x^3_4)^{s-\lambda} \mid f.$$

\textbf{Case 4.}    $3\lambda \leq \theta \leq 3s$.  Then $\max\{\lambda, \theta\}=\theta$. It is not hard to check that 
$$( x_1x^2_4)^{\theta}(x_3x^3_4)^{s-\theta} \mid f.$$
Consequently, we conclude that $f\in J^s$, as required. In particular, we get  
$$\mathrm{Ass}(J^s)=\mathrm{Ass}(J)=\{(x_1, x_3), (x_1, x_4)\}.$$
This yields that $J$ is NTF.  We therefore  obtain $I$ to be  normally torsion-free. On the other hand, using 
 {\it Macaulay2} \cite{GS}, we find that the polarization  of $I$ is  
 \begin{align*} 
 L:=\widetilde{I}=(&  z_{(0,0)}z_{(0,1)}z_{(0,2)}z_{(0,3)},z_{(0,0)}z_{(0,1)}z_{(0,2)}z_{(1,0)},z_{(0,0)}z_{(0,1)}z_{(1,0)}z_{(1,1)},\\
 &z_{(0,0)}z_{(1,0)}z_{(1,1)}z_{(1,2)},z_{(1,0)}z_{(1,1)}z_{(1,2)}z_{(1,3)},z_{(0,0)}z_{(0,1)}z_{(0,2)}z_{(3,0)},\\
& z_{(0,0)}z_{(0,1)}z_{(1,0)} z_{(3,0)},z_{(0,0)}z_{(1,0)}z_{(1,1)}z_{(3,0)},z_{(1,0)}z_{(1,1)}z_{(1,2)}z_{(3,0)},\\
 &  z_{(0,0)}z_{(0,1)}z_{(3,0)} z_{(3,1)},z_{(0,0)}z_{(1,0)}z_{(3,0)}z_{(3,1)},z_{(1,0)}z_{(1,1)}z_{(3,0)}z_{(3,1)},\\
& z_{(0,0)}z_{(2,0)}z_{(3,0)}z_{(3,1)}z_{(3,2)}, z_{(1,0)}z_{(2,0)}z_{(3,0)}z_{(3,1)}z_{(3,2)}),
\end{align*}
in       $S=K[z_{(0,0)},z_{(0,1)},z_{(0,2)},z_{(0,3)},z_{(1,0)},z_{(1,1)},z_{(1,2)},z_{(1,3)},z_{(2,0)},z_{(3,0)},z_{(3,1)},z_{(3,2)}].$ 
Using {\it Macaulay2} \cite{GS}, we detect that 
 \begin{align*}
 &\#\mathrm{Ass}(L)=19, ~ \#\mathrm{Ass}(L^2)=51,~\#\mathrm{Ass}(L^3)=85,~ \#\mathrm{Ass}(L^4)=126, \\
 & \#\mathrm{Ass}(L^5)= \#\mathrm{Ass}(L^6)=\#\mathrm{Ass}(L^7)=128.
 \end{align*}
 This means that  $\widetilde{I}$ is not normally torsion-free. 
 
 Now, our aim is to present a  monomial ideal $M$, which is not normally torsion-free, while  $\widetilde{M}$   is. 
 To do this, let  $M=(x_1^3x_2, x_1x_2^3, x_2^4, x_1^4x_3)$ in  $S=K[x_1, x_2, x_3]$. 
It is a straightforward operation to check that  $(x_1, x_2, x_3)\in \mathrm{Ass}(S/M^2)\setminus \mathrm{Ass}(S/M)$. 
This shows that  $M$ is NTF. 
  For convenience of notation, the polarization of $M$ is denoted by  
 $\widetilde{M}=(x_1x_2x_3x_5, x_1x_5x_6x_7, x_5x_6x_7x_8, x_1x_2x_3x_4x_9) \subset  K[x_1, \ldots, x_9].$ 
   Using  {\it Macaulay2} \cite{GS}, we deduce that 
 \begin{align*}
 \widetilde{M}&=(x_1,x_5)\cap (x_1,x_6)\cap  (x_1,x_7)\cap (x_1,x_8)\cap (x_2,x_5)\cap (x_2,x_6)\cap (x_2,x_7)\cap (x_3,x_5)\\
 &\cap (x_3,x_6)\cap (x_3,x_7)\cap (x_4,x_5)\cap  (x_5,x_9).
 \end{align*}
 Now, consider the following graph $G$. 
 
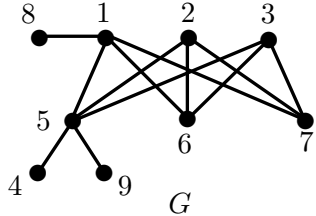
 \begin{center}

\scalebox{1.1}  
{
\begin{pspicture}(0,-1.3329687)(4.0428123,1.3329687)
\psdots[dotsize=0.2](1.3209375,0.83453125)
\psdots[dotsize=0.2](2.3209374,0.83453125)
\psdots[dotsize=0.2](3.2809374,0.81453127)
\psdots[dotsize=0.2](3.7209375,-0.16546875)
\psdots[dotsize=0.2](0.9209375,-0.16546875)
\psdots[dotsize=0.2](2.3009374,-0.14546876)
\psdots[dotsize=0.2](0.5009375,-0.78546876)
\psdots[dotsize=0.2](1.3009375,-0.78546876)
\psline[linewidth=0.04cm](0.9209375,-0.16546875)(0.5209375,-0.7654688)
\psline[linewidth=0.04cm](0.9009375,-0.16546875)(1.3009375,-0.78546876)
\psline[linewidth=0.04cm](0.9009375,-0.16546875)(1.3409375,0.83453125)
\psline[linewidth=0.04cm](1.3209375,0.8545312)(2.3009374,-0.12546875)
\psline[linewidth=0.04cm](1.3209375,0.87453127)(3.7009375,-0.14546876)
\psline[linewidth=0.04cm](3.2809374,0.81453127)(3.7009375,-0.12546875)
\psline[linewidth=0.04cm](2.3209374,0.83453125)(3.7009375,-0.12546875)
\psline[linewidth=0.04cm](2.3209374,-0.12546875)(3.2809374,0.81453127)
\psline[linewidth=0.04cm](0.9209375,-0.14546876)(2.3209374,0.83453125)
\psline[linewidth=0.04cm](2.3009374,0.81453127)(2.3009374,-0.10546875)
\psline[linewidth=0.04cm](0.9409375,-0.18546875)(0.9409375,-0.16546875)
\psline[linewidth=0.04cm](0.9209375,-0.16546875)(3.2609375,0.81453127)
\psdots[dotsize=0.2](0.5209375,0.83453125)
\psline[linewidth=0.04cm](0.5009375,0.8545312)(1.3209375,0.8545312)
\usefont{T1}{ptm}{m}{n}
\rput(0.39234376,1.1245313){$8$}
\usefont{T1}{ptm}{m}{n}
\rput(1.2923437,1.1445312){$1$}
\usefont{T1}{ptm}{m}{n}
\rput(2.3123438,1.1445312){$2$}
\usefont{T1}{ptm}{m}{n}
\rput(3.2723436,1.1245313){$3$}
\usefont{T1}{ptm}{m}{n}
\rput(0.57234377,-0.15546875){$5$}
\usefont{T1}{ptm}{m}{n}
\rput(2.2723436,-0.43546876){$6$}
\usefont{T1}{ptm}{m}{n}
\rput(3.7323437,-0.41546875){$7$}
\usefont{T1}{ptm}{m}{n}
\rput(0.23234375,-0.93546873){$4$}
\usefont{T1}{ptm}{m}{n}
\rput(1.5523437,-0.93546873){$9$}
\usefont{T1}{ptm}{m}{n}
\rput(2.2123437,-1.1554687){$G$}
\end{pspicture} 
}
  \end{center}
  
  It is straightforward to detect that $\widetilde{M}$ is the cover ideal of the graph $G$. Since $G$ has no odd cycle subgraph, this implies that $G$ is bipartite. It follows immediately  that the cover ideal of $G$ is NTF, and so $\widetilde{M}$ is normally torsion-free, as claimed.


\bigskip
\noindent{\bf Acknowledgments.}
First, the authors  are  deeply grateful to the anonymous referee for careful reading of the manuscript, and for  his/her  valuable suggestions which led to 
several  improvements in the quality of this paper. 
In addition, this paper was prepared when the first author  visited the Department of Mathematics of 
Uppsala University in 2023; in particular, he  would like to thank Uppsala 
University for its hospitality. Also, the second author was partially supported by the Lundstr\"om-\AA{m}an Foundation.


 
 \begin{itemize}
 \item \textsc{Mehrdad ~Nasernejad:  Univ. Artois, UR 2462, Laboratoire de Math\'{e}matique de  Lens (LML),   F-62300 Lens, France} 
 \item \textsc{Veronica Crispin Qui\~n{o}nez:  Department of Mathematics, Uppsala University, S-751 06, Uppsala, Sweden}
\item \textsc{Jonathan Toledo:  Tecnol\'o{g}ico Nacional de M\'e{x}ico, Instituto Tecnol\'o{g}ico Del Valle de Etla, Abasolo S/N, Barrio Del Agua Buena, Santiago
Suchilquitongo, 68230, Oaxaca, M\'e{x}ico}
 \end{itemize}

\end{document}